\documentclass[a4paper]{article}%
\usepackage{amssymb}
\usepackage{amsfonts}
\usepackage{amsmath}%
\usepackage{graphicx}
\providecommand{\U}[1]{\protect\rule{.1in}{.1in}}
\newtheorem{theorem}{Theorem}

\newtheorem{lemma}[theorem]{Lemma}

\newtheorem{proposition}[theorem]{Proposition}
\newtheorem{remark}[theorem]{Remark}

\newenvironment{proof}[1][Proof]{\noindent\textbf{#1.} }{\ \rule{0.5em}{0.5em}}
\numberwithin{equation}{section}
\begin{document}

\title{A Morita type proof of the replica-symmetric formula for SK}
\author{Erwin Bolthausen, University of Zurich}
\date{}
\maketitle

\begin{abstract}
We give a proof of the replica symmetric formula for the free energy of the
Sherrington-Kirkpatrick model in high temperature which is based on the TAP
formula. This is achieved by showing that the conditional annealed free energy
equals the quenched one, where the conditioning is given by an appropriate
$\sigma$-field with respect to which the TAP solutions are measurable.

\end{abstract}

\section{Introduction}

We consider the standard \textbf{Sherrington-Kirkpatrick model} with an
external field having the random Hamiltonian%
\begin{equation}
H_{\beta,h}\left(  \mathbf{\sigma}\right)  :=\frac{\beta}{\sqrt{2}}%
\sum_{i,j=1}^{N}g_{ij}^{\left(  N\right)  }\sigma_{i}\sigma_{j}+h\sum
_{i=1}^{N}\sigma_{i}\label{SK_Hamiltonian}%
\end{equation}
where $\beta>0$ and $h\in\mathbb{R}$ are real parameters,\ $\mathbf{\sigma
}=\left(  \sigma_{i}\right)  \in\Sigma_{N}:=\left\{  -1,1\right\}  ^{N},$ and
$g_{ij}^{\left(  N\right)  }$ for $i,j$ are i.i.d. centered Gaussians with
variance $1/N,$ defined on a probability space $\left(  \Omega,\mathcal{F}%
,\mathbb{P}\right)  $.

The random partition function is%
\[
Z_{N,\beta,h}:=2^{-N}\sum_{\mathbf{\sigma}}\exp\left[  H_{\beta,h}\left(
\mathbf{\sigma}\right)  \right]  ,
\]
and the Gibbs distribution is%
\begin{equation}
\operatorname*{GIBBS}\nolimits_{N,\beta,h}\left(  \mathbf{\sigma}\right)
:=\frac{2^{-N}}{Z_{N,\beta,h}}\exp\left[  H_{\beta,h}\left(  \mathbf{\sigma
}\right)  \right]  .\label{Gibbs}%
\end{equation}
It is known that%
\[
f\left(  \beta,h\right)  :=\lim_{N\rightarrow\infty}\frac{1}{N}\log
Z_{N,\beta,h}=\lim_{N\rightarrow\infty}\frac{1}{N}\mathbb{E}\log Z_{N,\beta,h}%
\]
exists, is non-random, and is given by the Parisi variational formula (see
\cite{GuTo},\ \cite{TalaBuch},\ \cite{PanchBook}). Furthermore, for small
$\beta,\ f\left(  \beta,h\right)  $ is given by the replica-symmetric formula,
originally proposed by Sherrington and Kirkpatrick (\cite{SK}):

\begin{theorem}
\label{Th_main}There exists $\beta_{0}>0$ such that for all $h,\beta$ with
$\beta\leq\beta_{0}$%
\[
f\left(  \beta,h\right)  =\operatorname*{RS}\left(  \beta,h\right)
:=\inf_{q\geq0}\left[  \int\log\cosh\left(  h+\beta\sqrt{q}x\right)
\phi\left(  dx\right)  +\frac{\beta^{2}\left(  1-q\right)  ^{2}}{4}\right]  .
\]
Here, $\phi$ is the standard Gaussian distribution.
\end{theorem}

For $h\neq0,$ the infimum is uniquely attained at $q=q\left(  \beta,h\right)
$ which satisfies%
\begin{equation}
q=\int\tanh^{2}\left(  h+\beta\sqrt{q}x\right)  \phi\left(  dx\right)
.\label{fixed_point}%
\end{equation}
This equation has a unique solution for $h\neq0,$ and for $h=0$ if $\beta
\leq1.$ For $\beta>1$ (and $h=0$), there are two solutions, one being $0,$ and
a positive one, which is the relevant for the minimization (see
\cite{TalaBuch}). We will assume $h>0,$ and $q$ will exclusively be used for
this number$.$

$f\left(  \beta,h\right)  =\operatorname*{RS}\left(  \beta,h\right)  $ is
believed to be true under the de Almeida-Thouless condition (AT-condition for
short)%
\begin{equation}
\beta^{2}\int\frac{\phi\left(  dx\right)  }{\cosh^{4}\left(  h+\beta\sqrt
{q}x\right)  }\leq1,\label{AT}%
\end{equation}
but this is still an open problem. At $h=0,$ the AT-condition is $\beta\leq1,$
and in this regime, $f\left(  \beta,0\right)  =\operatorname*{RS}\left(
\beta,0\right)  =\beta^{2}/4$ is known since long and can easily be proved by
a second moment method. In fact, in this case, the free energy equals the
annealed free energy%
\[
f\left(  \beta,0\right)  =f_{\mathrm{\mathrm{ann}}}\left(  \beta,0\right)
=\lim_{N\rightarrow\infty}\frac{1}{N}\log\mathbb{E}Z_{N,\beta,0}=\beta^{2}/4.
\]
It is however easy to see that for $h\neq0$, and any $\beta>0,$ neither
$f\left(  \beta,h\right)  $ nor $\operatorname*{RS}\left(  \beta,h\right)  $
equals $f_{\mathrm{\mathrm{ann}}}\left(  \beta,h\right)  $.

The aim of this note is to prove that $f\left(  \beta,h\right)
=\operatorname*{RS}\left(  \beta,h\right)  $ can, for small $\beta,$ be proved
by a \textit{conditional }\textquotedblleft
quenched=annealed\textquotedblright\ argument, via a second moment method.
Roughly speaking, we prove that there is a sub-$\sigma$-field$\mathcal{\ G}%
\subset\mathcal{F}$ such that $f\left(  \beta,h\right)  =\lim_{N\rightarrow
\infty}N^{-1}\log\mathbb{E}\left(  \left.  Z_{N}\right\vert \mathcal{G}%
\right)  =\operatorname*{RS}\left(  \beta,h\right)  $ almost surely, and where
we can estimate the conditional second moment by the square of the first one.
A key point is the connection of $\mathcal{G}$ (it will actually be a sequence
of $\sigma$-fields) with the Thouless-Anderson-Palmer equation, introduced in
\cite{TAP}, and in particular with the recursive construction given in
\cite{BoTAP}. The reason the method works is that the conditionally annealed
Gibbs measure is essentially a Curie-Weiss type model, centered at the
solution of the TAP equation, and as such it can be analyzed as a classical
mean-field model.

The method is closely related to arguments used for the first time by Morita
in \cite{Morita}. In fact, Morita invented the method to derive the quenched
free energy by a partial annealing, fixing part of the Hamilton which is
handled in a \textquotedblleft quenched way\textquotedblright, but where this
quenched part can be analyzed much easier than for the full Hamiltonian. This
is exactly what we do here by the conditioning.

Unfortunately, the argument does not seem to work in the full AT-region. This
is partly due to the fact that the second moment method does not work up to
the correct critical line. There are however also other difficulties.

Therefore, the result we prove is not new at all, and in fact, the proof is
quite longer than existing proofs. However, we believe the method is of
interest, and can be used quite broadly for other models.

A related approach has recently been developed independently by Jian Ding and
Nike Sun \cite{DingSun} for the lower bound of the memory capacity of a
version of the perceptron at zero temperature up to the predicted critical
value for the validity of the replica symmetric solution.

Our proof given does not use \textit{any }of the results on the SK-model
obtained previously, except for very simple ones, like the proof of the
uniqueness of $q$ for $h\neq0$, and on some simple computations from
\cite{BoTAP}. The core of the argument given here does not use the result from
\cite{BoTAP}, but it is motivated by the construction given there.

\noindent\textbf{Basic assumptions and notations: }We always assume $h\neq0,$
as there is nothing new in the argument for $h=0$ (but see the comments at the
end of the paper). For convenience, we assume $h>0.$ We usually drop the $N$
in $g_{ij}^{\left(  N\right)  },$ but the reader should keep in mind that
essentially any formula we write depends on the size parameter $N$. We also
often drop the parameters $\beta,h$ in the notation. If we write
\textquotedblleft for $\beta$ small enough\textquotedblright, we mean that
there exists $\beta_{0}>0$ such that the statement holds for $\beta\leq
\beta_{0}$ and for all $h>0.$ We will not be specific about $\beta_{0}$.

We typically use boldface letters, like $\mathbf{x}$, for vectors in
$\mathbb{R}^{N},$ occasionally random vectors, with components $x_{1}%
,\ldots,x_{N}.$ If $f:\mathbb{R\rightarrow R}$, we write $f\left(
\mathbf{x}\right)  \in\mathbb{R}^{N}$ for the vector with components $f\left(
x_{i}\right)  .$

In $\mathbb{R}^{N}$, we will use the inner product%
\[
\left\langle \mathbf{x},\mathbf{y}\right\rangle :=\frac{1}{N}\sum_{i=1}%
^{N}x_{i}y_{i},
\]
and the norm $\left\Vert \mathbf{x}\right\Vert :=\sqrt{\left\langle
\mathbf{x},\mathbf{x}\right\rangle }.$ We will also use the shorthand
$\operatorname*{Th}\left(  x\right)  :=\tanh\left(  h+\beta x\right)  .$

We use $Z,Z^{\prime},Z_{1}$ etc. for generic standard Gaussian random
variables. If several of them appear in one formula, then they are assumed to
be independent. We write $E$ for the expectation with respect to them.
\textquotedblleft Gaussian\textquotedblright\ always means centered Gaussian
unless stated otherwise. We hope the reader will not confuse these $Z$'s with
the partition functions, but it should always be clear from the context what
is what.

The Gibbs expectation under (\ref{Gibbs}) is usually written as $\left\langle
\cdot\right\rangle .$ $C$ is used as a generic positive constant which may
change from line to line.

If $\mathbf{a,\mathbf{b}}\in\mathbb{R}^{N},$ we write $\mathbf{a\otimes
\mathbf{b}}$ for the matrix%
\[
\left(  \mathbf{a\otimes\mathbf{b}}\right)  _{ij}:=\frac{a_{i}b_{j}}{N}.
\]
Remark that if $\mathbf{a,\mathbf{b,\mathbf{c}}}\in\mathbb{R}^{N},$ then
$\left(  \mathbf{\mathbf{a}}\otimes\mathbf{\mathbf{\mathbf{b}}}\right)
\mathbf{c=}\left\langle \mathbf{\mathbf{\mathbf{b}}},\mathbf{c}\right\rangle
\mathbf{a.}$

If $A$ is matrix, we write $A^{T}$ for the transposed, and if $A$ is square%
\[
\bar{A}:=\frac{1}{\sqrt{2}}\left(  A+A^{T}\right)  .
\]

\noindent\textbf{Outline of the argument: }We end the introduction with a
quick outline of the main idea. The Gibbs means $\mathbf{m}:=\left\langle
\mathbf{\sigma}\right\rangle $ are random variables. These random variables
satisfy (in the $N\rightarrow\infty$ limit) the so-called TAP equations. The
basic idea is to write the partition function $Z_{N}$ in terms of an average
over an appropriately tilted coin-tossing measure%
\[
p\left(  \mathbf{\sigma}\right)  =\prod_{i=1}^{N}p_{i}\left(  \sigma
_{i}\right)
\]
where%
\[
p_{i}\left(  \sigma_{i}\right)  =\frac{2^{-N}\mathrm{e}^{h_{i}\sigma_{i}}%
}{\cosh\left(  h_{i}\right)  },\ \sigma_{i}=\pm1,
\]
where $\mathbf{h}$ satisfies $\mathbf{m}=\tanh\left(  \mathbf{h}\right)  ,$
i.e. the expectation of $\sigma_{i}$ under $p_{i}$ is exactly $m_{i}$ where
$\mathbf{m}$ satisfies (approximately) the TAP equations%
\[
\mathbf{m}=\operatorname*{Th}\left(  \mathbf{\bar{g}m}-\beta\left(
1-q\right)  \mathbf{m}\right)  .
\]
Then%
\[
Z_{N}=\sum_{\mathbf{\sigma}}2^{-N}\exp\left[  H\left(  \mathbf{\sigma}\right)
\right]  =\prod_{i=1}^{N}\cosh\left(  h_{i}\right)  \sum_{\mathbf{\sigma}%
}p\left(  \mathbf{\sigma}\right)  \exp\left[  H\left(  \mathbf{\sigma}\right)
-N\left\langle \mathbf{h},\mathbf{\sigma}\right\rangle \right]  .
\]%
\[
\frac{1}{N}\log Z=\frac{1}{N}\sum_{i=1}^{N}\log\cosh\left(  h_{i}\right)
+\frac{1}{N}\log\sum_{\mathbf{\sigma}}p\left(  \mathbf{\sigma}\right)
\exp\left[  H\left(  \mathbf{\sigma}\right)  -N\left\langle \mathbf{h}%
,\mathbf{\sigma}\right\rangle \right]  .
\]
The a.s.-limit of the first part will be easy to evaluate, and gives%
\[
E\log\cosh\left(  h+\beta\sqrt{q}Z\right)
\]
which is the first part of the replica symmetric formula. For the second part,
we apply a variant of the second moment method, but it is quite delicate, as
the measures $p$ depend on the random variables $g_{ij}$. Therefore, we
construct a sub-$\sigma$-field $\mathcal{G}$ which has the property that
$\mathbf{m}$ is $\mathcal{G}$-m.b. Then one has%
\begin{multline*}
\mathbb{E}\left(  \left.  \sum\nolimits_{\sigma}p\left(  \sigma\right)
\exp\left[  H\left(  \sigma\right)  -N\left\langle \mathbf{h},\mathbf{\sigma
}\right\rangle \right]  \right\vert \mathcal{G}\right) \\
=\sum\nolimits_{\sigma}p\left(  \sigma\right)  \mathbb{E}\left(  \left.
\exp\left[  H\left(  \sigma\right)  -N\left\langle \mathbf{h},\mathbf{\sigma
}\right\rangle \right]  \right\vert \mathcal{G}\right)  ,
\end{multline*}
and it will turn out that $\mathbb{E}\left(  \left.  \exp\left[  H\left(
\sigma\right)  -N\left\langle \mathbf{h},\mathbf{\sigma}\right\rangle \right]
\right\vert \mathcal{G}\right)  \approx\exp\left[  N\beta^{2}\left(
1-q\right)  ^{2}/4\right]  $ for small $\beta$. Furthermore, one can estimate
the conditional second moment. The implementation of this idea requires not
one $\sigma$-field $\mathcal{G}$, but a sequence $\left\{  \mathcal{G}%
_{k}\right\}  .$

\section{The recursive modification of the interaction matrix}

We will not explicitly use the TAP equations, but the reader should keep in
mind the rough outline of the argument given above. In spirit, we will heavily
rely on the construction in \cite{BoTAP}, but we will not use in a substantial
way the results of this paper. For the purpose here, it is simpler to work
directly with random variables which are approximations of the iterative
scheme in \cite{BoTAP} which constructed approximations for the TAP equations
through%
\begin{equation}
\mathbf{m}^{\left(  k+1\right)  }:=\operatorname*{Th}\left(  \mathbf{\bar{g}%
m}^{\left(  k\right)  }-\beta\left(  1-q\right)  \mathbf{m}^{\left(
k-1\right)  }\right) \label{TAP_iteration}%
\end{equation}
with some initialization. We proved in \cite{BoTAP} that these random
variables defined through this iteration have a representation which makes it
possible to prove the convergence in the full high temperature region. We
directly use here this representation without using the iterative scheme
above. There is a further slight, but technically convenient, modification to
the approach in \cite{BoTAP}. There, we took the symmetrized matrix
$\mathbf{g=}\left(  g_{ij}\right)  $ which has i.i.d. Gaussian entries for
$i<j$ with variance $1/N,$ and $g_{ii}=0.$ Fixing the diagonal to be $0$ is of
course of no relevance as the diagonal part cancels out in the Gibbs
distribution. We then did construct a sequence $\mathbf{g}$ of modifications
$\mathbf{g}^{\left(  k\right)  }$, and a sequence $\mathcal{F}_{k}$ of
sub-$\sigma$-fields, whose behavior is the crucial part of the analysis. In
particular, the $\mathbf{g}^{\left(  k\right)  }$ are conditionally Gaussian,
given $\mathcal{F}_{k-2},$ and conditionally independent of $\mathcal{F}%
_{k-1}.$ Of crucial importance for the analysis in \cite{BoTAP} and also for
the analysis here is the behavior of the conditional covariances.
Unfortunately, the estimates for these in \cite{BoTAP} were quite complicated,
and we need them here still a bit more precise.

It turns out that these computations are simpler by sticking to $g_{ij}$ which
are independent for $i,j\leq N.$ The symmetrized matrix is then $\mathbf{\bar
{g}:=}\left(  \mathbf{g}+\mathbf{g}^{T}\right)  /\sqrt{2}.$ This looks being a
trivial rewriting, but we will define the $\sigma$-fields $\mathcal{G}_{k}$
here in terms of $\mathbf{g}$, and therefore, they are different from the
$\mathcal{F}_{k}$ used in \cite{BoTAP}. The main advantage is that the
construction of the $\mathbf{g}^{\left(  k\right)  }$ is explicit for all $k,$
and the conditional covariances we need are totally explicit as well, which
simplifies the computations considerably.\footnote{Unfortunately, we
overlooked this simplification when writing \cite{BoTAP} which would have
saved 1/3 of that paper.}

We construct sequence $\left\{  \gamma_{k}\right\}  _{k\geq1}$, $\left\{
\rho_{k}\right\}  _{k\geq1}$ of real numbers, and sequences of random matrices
$\mathbf{g}^{\left(  k\right)  }$ together with sequences of random vectors
$\mathbf{\phi}^{\left(  k\right)  }\in\mathbb{R}^{N},\ k\geq1.$ Define%
\[
\gamma_{1}=E\tanh\left(  h+\beta\sqrt{q}Z\right)  ,\ \rho_{1}:=\sqrt{q}%
\gamma_{1},
\]
and recursively%
\[
\rho_{k}:=\psi\left(  \rho_{k-1}\right)  ,~\gamma_{k}:=\frac{\rho_{k}%
-\sum_{j=1}^{k-1}\gamma_{j}^{2}}{\sqrt{q-\sum_{j=1}^{k-1}\gamma_{j}^{2}}}%
\]
where $\psi:\left[  0,q\right]  \rightarrow(0,q]$ is defined by%
\[
\psi\left(  t\right)  :=E\operatorname*{Th}\left(  \sqrt{t}Z+\sqrt
{q-t}Z^{\prime}\right)  \operatorname*{Th}\left(  \sqrt{t}Z+\sqrt
{q-t}Z^{\prime\prime}\right)  .
\]
Remark that $\psi\left(  q\right)  =q,$ and $\psi\left(  0\right)  =\gamma
_{1}^{2}.$ The following easy result was proved in \cite{BoTAP}.

\begin{lemma}
\label{Le_sequences&AT}

\begin{enumerate}
\item[a)] $\left\{  \rho_{k}\right\}  $ is an increasing sequence of positive
numbers. $\lim_{k\rightarrow\infty}\rho_{k}=q$ holds if and only if (\ref{AT})
is satisfied. If (\ref{AT}) holds with the strict inequality, then the
convergence of $\left\{  \rho_{k}\right\}  $ is exponentially fast.

\item[b)] $\Gamma_{k-1}^{2}:=\sum_{j=1}^{k-1}\gamma_{j}^{2}<\rho_{k}<q$ holds
for all $k,$ and $\sum_{j=1}^{\infty}\gamma_{j}^{2}=q$ holds if and only if
(\ref{AT}) is satisfied.
\end{enumerate}
\end{lemma}

Next, we define the recursions for $\mathbf{g}^{\left(  k\right)
},\ \mathbf{\phi}^{\left(  k\right)  }.$ It is convenient to also introduce
vectors $\mathbf{h}^{\left(  k\right)  },\ \mathbf{m}^{\left(  k\right)  }$
which are directly related to the $\mathbf{\phi}$'s. (The $\mathbf{m}^{\left(
k\right)  }$ are the approximate solutions of the TAP equations). For $k=1$:%
\[
\mathbf{g}^{\left(  1\right)  }:=\mathbf{g},\ \mathbf{\phi}^{\left(  1\right)
}:=\mathbf{1,\ \mathbf{m}}^{\left(  1\right)  }:=\sqrt{q}\mathbf{1.}%
\]
Assume that $\mathbf{g}^{\left(  s\right)  },\ \mathbf{\phi}^{\left(
s\right)  },\mathbf{m}^{\left(  s\right)  }$ are defined for $s\leq k.$ Set%
\begin{align}
\mathbf{\xi}^{\left(  s\right)  }  & :=\mathbf{g}^{\left(  s\right)
}\mathbf{\phi}^{\left(  s\right)  },\ \mathbf{\eta}^{\left(  s\right)
}:=\mathbf{g}^{\left(  s\right)  T}\mathbf{\phi}^{\left(  s\right)
},\label{xik_etak_zetak_Def}\\
\mathbf{\zeta}^{\left(  s\right)  }  & :=\frac{\mathbf{\xi}^{\left(  s\right)
}+\mathbf{\eta}^{\left(  s\right)  }}{\sqrt{2}}=\overline{\mathbf{g}^{\left(
s\right)  }}\mathbf{\phi}^{\left(  s\right)  },\nonumber
\end{align}
and we write%
\begin{equation}
\mathcal{G}_{k}:=\sigma\left(  \mathbf{\xi}^{\left(  m\right)  },\mathbf{\eta
}^{\left(  m\right)  }:m\leq k\right)  .\label{Fk}%
\end{equation}
We will write $\mathbb{E}_{k}$ for the conditional expectation with respect to
to $\mathcal{G}_{k}.$ Remark that $\left\langle \mathbf{\phi}^{\left(
k\right)  },\mathbf{\xi}^{\left(  k\right)  }\right\rangle =\left\langle
\mathbf{\eta}^{\left(  k\right)  },\mathbf{\phi}^{\left(  k\right)
}\right\rangle .$

Put first%
\begin{equation}
\mathbf{h}^{\left(  k+1\right)  }:=h\mathbf{1}+\beta\sum\nolimits_{s=1}%
^{k-1}\gamma_{s}\mathbf{\zeta}^{\left(  s\right)  }+\beta\sqrt{q-\Gamma
_{k-1}^{2}}\mathbf{\zeta}^{\left(  k\right)  },\label{hk_Def}%
\end{equation}%
\begin{equation}
\mathbf{m}^{\left(  k+1\right)  }=\tanh\left(  \mathbf{h}^{\left(  k+1\right)
}\right)  .\label{mk_construktion}%
\end{equation}
We haven't defined $\mathbf{h}^{\left(  1\right)  },$ but we could put it
$\tanh^{-1}\left(  \sqrt{q}\right)  .$

We next define%
\begin{equation}
\mathbf{\phi}^{\left(  k+1\right)  }:=\frac{\mathbf{m}^{\left(  k+1\right)
}-\sum_{s=1}^{k}\left\langle \mathbf{m}^{\left(  k+1\right)  },\mathbf{\phi
}^{\left(  s\right)  }\right\rangle \mathbf{\phi}^{\left(  s\right)  }%
}{\left\Vert \mathbf{m}^{\left(  k+1\right)  }-\sum_{s=1}^{k}\left\langle
\mathbf{m}^{\left(  k+1\right)  },\mathbf{\phi}^{\left(  s\right)
}\right\rangle \mathbf{\phi}^{\left(  s\right)  }\right\Vert }%
.\label{phik_construction}%
\end{equation}
This requires that the denominator is $\neq0$ which is true with probability
$1$ (Lemma \ref{Le_denom}), assuming $N>k$. Finally%
\begin{equation}
\mathbf{g}^{\left(  k+1\right)  }:=\mathbf{g}^{\left(  k\right)
}-\mathbf{\rho}^{\left(  k\right)  },\label{gk_construction}%
\end{equation}
with%
\begin{equation}
\mathbf{\rho}^{\left(  k\right)  }:=\mathbf{\xi}^{\left(  k\right)  }%
\otimes\mathbf{\phi}^{\left(  k\right)  }+\mathbf{\phi}^{\left(  k\right)
}\otimes\mathbf{\eta}^{\left(  k\right)  }\mathbf{-}\left\langle \mathbf{\phi
}^{\left(  k\right)  },\mathbf{\xi}^{\left(  k\right)  }\right\rangle \left(
\mathbf{\phi}^{\left(  k\right)  }\otimes\mathbf{\phi}^{\left(  k\right)
}\right)  .\label{rho}%
\end{equation}

\begin{lemma}
\label{Le_basic1}

\begin{enumerate}
\item[a)] $\left\Vert \mathbf{\phi}^{\left(  k\right)  }\right\Vert =1$ for
all $k,$ and $\left\langle \mathbf{\phi}^{\left(  k\right)  },\mathbf{\phi
}^{\left(  t\right)  }\right\rangle =0$ for $k\neq t.$

\item[b)] For $s<k,$ one has $\mathbf{g}^{\left(  k\right)  }\mathbf{\phi
}^{\left(  s\right)  }=\mathbf{0}$, and $\mathbf{g}^{\left(  k\right)
T}\mathbf{\phi}^{\left(  s\right)  }=\mathbf{0}$.

\item[c)] $\mathbf{m}^{\left(  k\right)  }$ and $\mathbf{\phi}^{\left(
k\right)  }$ are $\mathcal{G}_{k-1}$-m.b. for all $k\geq1$.
\end{enumerate}
\end{lemma}

\begin{proof}
a) is evident by the definition.

b) We use induction on $k.$ For $k=1,$ there is nothing to prove. For $k=2,$
one just has to check that $\mathbf{g}^{\left(  2\right)  }\mathbf{1=0}%
$,\ $\mathbf{g}^{\left(  2\right)  T}\mathbf{1}=0,$ which are straightforward.
So, we assume $k\geq3.$ If $s=k-1$, using $\left\langle \mathbf{\phi}^{\left(
k-1\right)  },\mathbf{\phi}^{\left(  k-1\right)  }\right\rangle =1$%
\begin{align*}
\mathbf{g}^{\left(  k\right)  }\mathbf{\phi}^{\left(  k-1\right)  }  &
=\mathbf{g}^{\left(  k-1\right)  }\mathbf{\phi}^{\left(  k-1\right)
}-\mathbf{\rho}^{\left(  k-1\right)  }\mathbf{\phi}^{\left(  k-1\right)  }\\
& =\mathbf{\xi}^{\left(  k-1\right)  }-\mathbf{\xi}^{\left(  k-1\right)
}-\left\langle \mathbf{\eta}^{\left(  k-1\right)  },\mathbf{\phi}^{\left(
k-1\right)  }\right\rangle \mathbf{\phi}^{\left(  k-1\right)  }\\
& +\left\langle \mathbf{\eta}^{\left(  k-1\right)  },\mathbf{\phi}^{\left(
k-1\right)  }\right\rangle \mathbf{\phi}^{\left(  k-1\right)  }\\
& =0.
\end{align*}
If $s\leq k-2,$ we have by induction%
\[
\mathbf{g}^{\left(  k\right)  }\mathbf{\phi}^{\left(  s\right)  }%
=-\mathbf{\rho}^{\left(  k-1\right)  }\mathbf{\phi}^{\left(  s\right)  },
\]
and using $\left\langle \mathbf{\phi}^{\left(  k-1\right)  },\mathbf{\phi
}^{\left(  s\right)  }\right\rangle =0,$ and again induction, we have%
\begin{align*}
\mathbf{\rho}^{\left(  k-1\right)  }\mathbf{\phi}^{\left(  s\right)  }  &
=\mathbf{\xi}^{\left(  k-1\right)  }\left\langle \mathbf{\phi}^{\left(
k-1\right)  },\mathbf{\phi}^{\left(  s\right)  }\right\rangle +\mathbf{\phi
}^{\left(  k-1\right)  }\left\langle \mathbf{\phi}^{\left(  k-1\right)
},\mathbf{g}^{\left(  k-1\right)  }\mathbf{\phi}^{\left(  s\right)
}\right\rangle \\
& -\left\langle \mathbf{\phi}^{\left(  k-1\right)  },\mathbf{g}^{\left(
k-1\right)  }\mathbf{\phi}^{\left(  k-1\right)  }\right\rangle \left\langle
\mathbf{\phi}^{\left(  k-1\right)  },\mathbf{\phi}^{\left(  s\right)
}\right\rangle \mathbf{\phi}^{\left(  k-1\right)  }\\
& =0.
\end{align*}
$\mathbf{\phi}^{\left(  s\right)  }\mathbf{g}^{\left(  k\right)  }=\mathbf{0}$
is proved similarly.

c) It suffices to check that for $\mathbf{m}^{\left(  k\right)  }.$ As
$\mathbf{\zeta}^{\left(  s\right)  }$ is $\mathcal{G}_{s}$-m.b. for $s\leq
k-1, $ the claim follows.
\end{proof}

The motivation for the construction of $\mathbf{g}^{\left(  k\right)  }$ in
the form given in (\ref{gk_construction}) is the following

\begin{proposition}
\label{Prop_main_gk}

\begin{enumerate}
\item[a)] Conditionally on $\mathcal{G}_{k-2},\ \mathbf{g}^{\left(  k\right)
} $ and $\mathbf{g}^{\left(  k-1\right)  }$ are Gaussian. The conditional
covariances of $\mathbf{g}^{\left(  k\right)  }$ given $\mathcal{G}_{k-2}$ are
given by%
\begin{equation}
\mathbb{E}_{k-2}\left(  g_{ij}^{\left(  k\right)  }g_{st}^{\left(  k\right)
}\right)  =\frac{1}{N}\left[  \delta_{is}-\alpha_{is}^{\left(  k-1\right)
}\right]  \left[  \delta_{jt}-\alpha_{jt}^{\left(  k-1\right)  }\right]
,\label{g_covariances}%
\end{equation}
with the abbreviation%
\[
\alpha_{ij}^{\left(  m\right)  }:=\frac{1}{N}\sum\nolimits_{r=1}^{m}\phi
_{i}^{\left(  r\right)  }\phi_{j}^{\left(  r\right)  }.
\]
(By Lemma \ref{Le_basic1} c), $\alpha^{\left(  k-1\right)  }$ is
$\mathcal{G}_{k-2}$-m.b.)

\item[b)] Conditionally on $\mathcal{G}_{k-2},$ $\mathbf{g}^{\left(  k\right)
}$ is independent of $\mathcal{G}_{k-1}.$

\item[c)] The variables $\mathbf{\zeta}^{\left(  k\right)  }$ are
conditionally Gaussian, given $\mathcal{G}_{k-1}$ with covariances%
\begin{equation}
\mathbb{E}_{k-1}\zeta_{i}^{\left(  k\right)  }\zeta_{j}^{\left(  k\right)
}=\delta_{ij}+\frac{1}{N}\phi_{i}^{\left(  k\right)  }\phi_{j}^{\left(
k\right)  }-\alpha_{ij}^{\left(  k-1\right)  }\label{zetaij_Cov}%
\end{equation}

\end{enumerate}
\end{proposition}

\begin{proof}
We use the following induction scheme to prove a) and b):

\begin{itemize}
\item[(i)] We assume that the statements a), b) are correct for $k.$

\item[(ii)] b) implies trivially that $\mathbf{g}^{\left(  k\right)  }$ is
Gaussian conditionally on $\mathcal{G}_{k-1}.$ So, this part of a) for $k+1$
is already settled.

\item[(iii)] As $\mathbf{\phi}^{\left(  k\right)  }$ is $\mathcal{G}_{k-1}%
$-m.b., it follows that $\mathbf{\xi}^{\left(  k\right)  },\mathbf{\eta
}^{\left(  k\right)  }$ are Gaussian, conditionally on $\mathcal{G}_{k-1},$
simply because they are linear combinations of the $g_{ij}^{\left(  k\right)
} $ with coefficients which are $\mathcal{G}_{k-1}$-m.b.

\item[(iv)] From the form of $\mathbf{\rho}^{\left(  k\right)  },$ it then
follows that it is also Gaussian, conditionally on $\mathcal{G}_{k-1},$ and
therefore, $\mathbf{g}^{\left(  k+1\right)  }$ is Gaussian, conditionally on
$\mathcal{G}_{k-1}.$

\item[(v)] The rest is just a covariance check: In order to prove that
$\mathbf{g}^{\left(  k+1\right)  }$ is independent of $\mathcal{G}_{k}%
=\sigma\left(  \mathcal{G}_{k-1},\mathbf{\xi}^{\left(  k\right)
},\mathbf{\eta}^{\left(  k\right)  }\right)  ,$ we have to check that the
conditional covariances between $\mathbf{g}^{\left(  k+1\right)  }$ and
$\mathbf{\xi}^{\left(  k\right)  }$ given $\mathcal{G}_{k-1},$ and between
$\mathbf{g}^{\left(  k+1\right)  }$ and $\mathbf{\eta}^{\left(  k\right)  },$
vanish, which in fact heavily uses (\ref{g_covariances}) for $k$. Finally we
have to boost this formula to $k+1.$
\end{itemize}

We first have the compute the conditional covariances among the $\xi^{\left(
k\right)  }$'s and $\eta^{\left(  k\right)  }$'s.
\begin{align}
\mathbb{E}_{k-1}\xi_{i}^{\left(  k\right)  }\xi_{j}^{\left(  k\right)  }  &
=\sum_{s,t}\phi_{s}^{\left(  k\right)  }\phi_{t}^{\left(  k\right)
}\mathbb{E}_{k-1}g_{is}^{\left(  k\right)  }g_{jt}^{\left(  k\right)
}\nonumber\\
& =\frac{1}{N}\sum_{s,t}\phi_{s}^{\left(  k\right)  }\phi_{t}^{\left(
k\right)  }\left[  \delta_{ij}-\alpha_{ij}^{\left(  k-1\right)  }\right]
\left[  \delta_{st}-\alpha_{st}^{\left(  k-1\right)  }\right] \label{xi_cov}\\
& =\delta_{ij}-\alpha_{ij}^{\left(  k-1\right)  },\nonumber
\end{align}
and symmetrically the same for $\mathbb{E}_{k-1}\eta_{i}^{\left(  k\right)
}\eta_{j}^{\left(  k\right)  }.$%
\begin{align}
\mathbb{E}_{k-1}\xi_{i}^{\left(  k\right)  }\eta_{j}^{\left(  k\right)  }  &
=\sum_{s,t}\phi_{s}^{\left(  k\right)  }\phi_{t}^{\left(  k\right)
}\mathbb{E}_{k-1}g_{is}^{\left(  k\right)  }g_{tj}^{\left(  k\right)
}\nonumber\\
& =\sum_{s,t}\phi_{s}^{\left(  k\right)  }\phi_{t}^{\left(  k\right)  }%
\frac{1}{N}\left[  \delta_{it}-\alpha_{it}^{\left(  k-1\right)  }\right]
\left[  \delta_{sj}-\alpha_{sj}^{\left(  k-1\right)  }\right]
\label{xi_eta_cov}\\
& =\frac{1}{N}\phi_{i}^{\left(  k\right)  }\phi_{j}^{\left(  k\right)
}\nonumber
\end{align}

Let's next check that the covariances between $\mathbf{g}^{\left(  k+1\right)
}$ and $\mathbf{\xi}^{\left(  k\right)  }$ vanish:%
\[
\mathbb{E}_{k-1}g_{ij}^{\left(  k+1\right)  }\xi_{s}^{\left(  k\right)
}=\mathbb{E}_{k-1}g_{ij}^{\left(  k\right)  }\xi_{s}^{\left(  k\right)
}-\mathbb{E}_{k-1}\rho_{ij}^{\left(  k\right)  }\xi_{s}^{\left(  k\right)  }%
\]%
\begin{align*}
\mathbb{E}_{k-1}g_{ij}^{\left(  k\right)  }\xi_{s}^{\left(  k\right)  }  &
=\mathbb{E}_{k-1}g_{ij}^{\left(  k\right)  }\sum_{t}g_{st}^{\left(  k\right)
}\phi_{t}^{\left(  k\right)  }\\
& =\sum_{t}\phi_{t}^{\left(  k\right)  }\mathbb{E}_{k-1}g_{ij}^{\left(
k\right)  }g_{st}^{\left(  k\right)  }=\sum_{t}\phi_{t}^{\left(  k\right)
}\mathbb{E}_{k-2}g_{ij}^{\left(  k\right)  }g_{st}^{\left(  k\right)  }\\
& =\frac{1}{N}\left[  \delta_{is}-\alpha_{is}^{\left(  k-1\right)  }\right]
\sum_{t}\phi_{t}^{\left(  k\right)  }\left[  \delta_{jt}-\alpha_{jt}^{\left(
k-1\right)  }\right] \\
& =\frac{1}{N}\left[  \delta_{is}-\alpha_{is}^{\left(  k-1\right)  }\right]
\phi_{j}^{\left(  k\right)  }%
\end{align*}%
\begin{align*}
\mathbb{E}_{k-1}\rho_{ij}^{\left(  k\right)  }\xi_{s}^{\left(  k\right)  }  &
=\frac{1}{N}\phi_{j}^{\left(  k\right)  }\mathbb{E}_{k-1}\left(  \xi
_{s}^{\left(  k\right)  }\xi_{i}^{\left(  k\right)  }\right)  +\frac{1}{N}%
\phi_{i}^{\left(  k\right)  }\mathbb{E}_{k-1}\left(  \xi_{s}^{\left(
k\right)  }\eta_{j}^{\left(  k\right)  }\right) \\
& \mathbf{-}\phi_{i}^{\left(  k\right)  }\phi_{j}^{\left(  k\right)  }\frac
{1}{N^{2}}\sum_{u}\phi_{u}^{\left(  k\right)  }\mathbb{E}_{k-1}\xi
_{u}^{\left(  k\right)  }\xi_{s}^{\left(  k\right)  }\\
& =\frac{1}{N}\left(  \delta_{is}-\alpha_{is}^{\left(  k-1\right)  }\right)
\phi_{j}^{\left(  k\right)  }+\frac{1}{N}\phi_{i}^{\left(  k\right)  }\left(
\frac{1}{N}\phi_{s}^{\left(  k\right)  }\phi_{j}^{\left(  k\right)  }\right)
\\
& \mathbf{-}\phi_{i}^{\left(  k\right)  }\phi_{j}^{\left(  k\right)  }\frac
{1}{N^{2}}\sum_{u}\phi_{u}^{\left(  k\right)  }\left(  \delta_{us}-\alpha
_{us}^{\left(  k-1\right)  }\right) \\
& =\frac{1}{N}\left(  \delta_{is}-\alpha_{is}^{\left(  k-1\right)  }\right)
\phi_{j}^{\left(  k\right)  }.
\end{align*}
Therefore, $\mathbb{E}_{k-1}g_{ij}^{\left(  k\right)  }\xi_{s}^{\left(
k\right)  }=0,$ and similarly (and symmetrically) $\mathbb{E}_{k-1}%
g_{ij}^{\left(  k\right)  }\eta_{s}^{\left(  k\right)  }=0$ for all $i,j,s.$
So, this proves that the $\mathcal{G}_{k-1}$-conditional covariances between
$\mathbf{g}^{\left(  k+1\right)  }$ and $\left(  \mathbf{\xi}^{\left(
k\right)  },\mathbf{\eta}^{\left(  k\right)  }\right)  $ vanish which implies
that $\mathbf{g}^{\left(  k+1\right)  }$ is conditionally independent of
$\mathcal{G}_{k}$ given $\mathcal{G}_{k-1},$ as everything is conditionally Gaussian.

As a consequence, we also have%
\begin{equation}
\mathbb{E}_{k-1}\left(  \rho_{ij}^{\left(  k\right)  }g_{st}^{\left(
k\right)  }\right)  =0,\ \forall i,j,s,t.\label{rho_g_uncorr}%
\end{equation}

To finish the induction, it remains to prove the validity of
(\ref{g_covariances}) with $k$ replaced by $k+1.$ Using (\ref{rho_g_uncorr}),
one has%
\begin{align}
\mathbb{E}_{k-1}\left(  g_{ij}^{\left(  k+1\right)  }g_{st}^{\left(
k+1\right)  }\right)   & =\mathbb{E}_{k-1}\left(  \left[  g_{ij}^{\left(
k\right)  }-\rho_{ij}^{\left(  k\right)  }\right]  \left[  g_{st}^{\left(
k\right)  }-\rho_{st}^{\left(  k\right)  }\right]  \right) \label{g_k+1_cov1}%
\\
& =\mathbb{E}_{k-2}\left(  g_{ij}^{\left(  k\right)  }g_{st}^{\left(
k\right)  }\right)  +\mathbb{E}_{k-1}\left(  \rho_{ij}^{\left(  k\right)
}\rho_{st}^{\left(  k\right)  }\right)  .\nonumber
\end{align}%
\begin{align*}
\mathbb{E}_{k-1}\left(  \rho_{ij}^{\left(  k\right)  }\rho_{st}^{\left(
k\right)  }\right)   & =\frac{1}{N^{2}}\mathbb{E}_{k-1}\Bigg \{\left(  \xi
_{i}^{\left(  k\right)  }\phi_{j}^{\left(  k\right)  }+\phi_{i}^{\left(
k\right)  }\eta_{j}^{\left(  k\right)  }-\phi_{i}^{\left(  k\right)  }\phi
_{j}^{\left(  k\right)  }\frac{1}{N}\sum_{u}\phi_{u}^{\left(  k\right)  }%
\xi_{u}^{\left(  k\right)  }\right) \\
& \times\left(  \xi_{s}^{\left(  k\right)  }\phi_{t}^{\left(  k\right)  }%
+\phi_{s}^{\left(  k\right)  }\eta_{t}^{\left(  k\right)  }-\phi_{s}^{\left(
k\right)  }\phi_{t}^{\left(  k\right)  }\frac{1}{N}\sum_{u}\phi_{u}^{\left(
k\right)  }\xi_{u}^{\left(  k\right)  }\right)  \Bigg \}\\
& =\frac{1}{N^{2}}\phi_{j}^{\left(  k\right)  }\phi_{t}^{\left(  k\right)
}\left(  \delta_{is}-\alpha_{is}^{\left(  k-1\right)  }\right)  +\frac
{1}{N^{2}}\phi_{i}^{\left(  k\right)  }\phi_{s}^{\left(  k\right)  }\left(
\delta_{jt}-\alpha_{jt}^{\left(  k-1\right)  }\right) \\
& -\frac{1}{N^{3}}\phi_{i}^{\left(  k\right)  }\phi_{j}^{\left(  k\right)
}\phi_{s}^{\left(  k\right)  }\phi_{t}^{\left(  k\right)  }.
\end{align*}
Plugging that into (\ref{g_k+1_cov1}), and using (\ref{g_covariances}) for
$k,$ one gets it for $k+1.$ So, we have proved a) and b). c) follows from
(\ref{xi_cov}) and (\ref{xi_eta_cov}).
\end{proof}

\begin{lemma}
\label{Le_denom}For all $k,$ and $N>k$%
\[
\left\Vert \mathbf{m}^{\left(  k+1\right)  }-\sum\nolimits_{s=1}%
^{k}\left\langle \mathbf{m}^{\left(  k+1\right)  },\mathbf{\phi}^{\left(
s\right)  }\right\rangle \mathbf{\phi}^{\left(  s\right)  }\right\Vert
>0,\ \mathbb{P}-\mathrm{a.s.}%
\]

\end{lemma}

\begin{proof}
We use induction on $k.$ For $k=0,$ there is nothing to prove, and $k=1$ is
evident, so we assume $k\geq2,$ and that $\mathbf{\phi}^{\left(  s\right)
},\ s\leq k$ is well-defined, and we can use the covariance computation in
Proposition \ref{Prop_main_gk} c). We prove that%
\[
\mathbb{P}_{k-1}\left(  \left\Vert \mathbf{m}^{\left(  k+1\right)  }%
-\sum\nolimits_{s=1}^{k}\left\langle \mathbf{m}^{\left(  k+1\right)
},\mathbf{\phi}^{\left(  s\right)  }\right\rangle \mathbf{\phi}^{\left(
s\right)  }\right\Vert >0\right)  =1,\ \mathbb{P}-\mathrm{a.s.}%
\]
In the expression (\ref{hk_Def}) of $\mathbf{h}^{\left(  k+1\right)  },$ all
the entries are $\mathcal{G}_{k-1}$-m.b. except $\mathbf{\zeta}^{\left(
k\right)  },$ and $q-\Gamma_{k-1}^{2}>0.$ Therefore, conditionally on
$\mathcal{G}_{k-1},$ we have%
\[
m_{i}^{\left(  k+1\right)  }=\tanh\left(  x_{i}+\alpha\zeta_{i}^{\left(
k\right)  }\right)
\]
with $x_{i}\in\mathbb{R},\ \alpha>0.$ From (\ref{zetaij_Cov}), the conditional
distribution of $\mathbf{\zeta}^{\left(  k\right)  }$ is Gaussian with a
covariance matrix of rank $N-k.$ From that, it is immediate that $\mathbb{P}%
$-a.s. there exists $i\leq N$ with $m_{i}^{\left(  k+1\right)  }$ having a
non-degenerate conditional distribution under $\mathbb{P}_{k-1}.$ This implies
the claim.
\end{proof}

For the formulation of the next result, we introduce the following notation.
If $X_{N},Y_{N}$ are two sequences of random variables, depending possibly on
other parameters like $\beta,h,k$ etc., we write%
\[
X_{N}\simeq Y_{N}%
\]
if there exists a constant $C>0,$ depending possibly on these other
parameters, but not on $N,$ with%
\[
\mathbb{P}\left(  \left\vert X_{N}-Y_{N}\right\vert \geq t\right)  \leq
C\exp\left[  -Ct^{2}N\right]  .
\]
$X_{N}\simeq Y_{N}$ in particular implies $\left\Vert X_{N}-Y_{N}\right\Vert
_{p}\rightarrow0$ for every $p\geq1$ as $N\rightarrow\infty.$

\begin{proposition}
\label{Prop_inner_products}

\begin{enumerate}
\item[a)] For any $j<k,$ one has%
\begin{equation}
\left\langle \mathbf{m}^{\left(  k\right)  },\mathbf{\phi}^{\left(  j\right)
}\right\rangle \simeq\gamma_{j}.\label{ipr_m_phi}%
\end{equation}

\item[b)] For any $k$%
\begin{equation}
\left\Vert \mathbf{m}^{\left(  k\right)  }\right\Vert ^{2}\simeq
q,\label{ipr_mm1}%
\end{equation}
and for $j<k$%
\begin{equation}
\left\langle \mathbf{m}^{\left(  k\right)  },\mathbf{m}^{\left(  j\right)
}\right\rangle \simeq\rho_{j}.\label{ipr_mm2}%
\end{equation}

\end{enumerate}
\end{proposition}

\begin{proof}
This was proved in \cite{BoTAP}. The $\mathbf{m}^{\left(  k\right)  }$ there
were defined through the iteration (\ref{TAP_iteration}), and we proved that
these random variables can be approximated by the ones essentially given by
(\ref{mk_construktion}). However, we have here a slightly different version,
as our $\mathcal{G}_{k}$ are not the same as the $\mathcal{F}_{k}$ in
\cite{BoTAP}. Therefore, we give a sketch of the proof here again.

a) is a simple consequence of b), see \cite{BoTAP}, Lemma 2.7.

So, we prove b). We first prove (\ref{ipr_mm1}). (\ref{ipr_mm2}) will be
proved by a small modification of the argument. $k=1$ is trivial, and%
\[
m_{i}^{\left(  2\right)  }=\operatorname*{Th}\left(  \sqrt{q}\zeta
_{i}^{\left(  1\right)  }\right)  ,
\]
and then (\ref{ipr_mm1}) follows from the LLN and the fixed point equation for
$q.$ So, we assume $k\geq3.$ We have%
\[
m_{i}^{\left(  k\right)  }=\operatorname*{Th}\left(  \sum\nolimits_{s=1}%
^{k-2}\gamma_{s}\zeta_{i}^{\left(  s\right)  }+\sqrt{q-\Gamma_{k-2}^{2}}%
\zeta_{i}^{\left(  k-1\right)  }\right)
\]
We observe that $\operatorname*{Th}\left(  x+\cdot\right)  $ is Lipshitz
continuous with $\left\Vert \operatorname*{Th}\left(  x+\cdot\right)
\right\Vert _{\mathrm{lip}}=\max\left(  1,\beta\right)  $ for any
$x\in\mathbb{R}$. We consider now the conditional distribution of
$m_{i}^{\left(  k\right)  }$ with respect to $\mathcal{G}_{k-2}.$ The Lipshitz
norm of $x\longmapsto\operatorname*{Th}\left(  \sum\nolimits_{s=1}^{k-2}%
\gamma_{s}\zeta_{i}^{\left(  s\right)  }+\sqrt{q-\Gamma_{k-2}^{2}}x\right)  $
is $\max\left(  1,\beta\sqrt{q-\Gamma_{k-2}^{2}}\right)  .$ As
$\operatorname*{Th}$ is bounded by $1,$ we have that the Lipshitz norm of
$x\longmapsto\operatorname*{Th}^{2}\left(  \sum\nolimits_{s=1}^{k-2}\gamma
_{s}\zeta_{i}^{\left(  s\right)  }+\sqrt{q-\Gamma_{k-2}^{2}}x\right)  $ is
bounded by $2\max\left(  1,\beta\sqrt{q-\Gamma_{k-2}^{2}}\right)  .$ Applying
Lemma \ref{Le_chi}, and the conditional covariances of $\mathbf{\zeta
}^{\left(  k-1\right)  }$ given in Proposition \ref{Prop_main_gk} above, we
obtain%
\begin{multline*}
\mathbb{P}_{k-2}\left(  \left\vert \frac{1}{N}\sum_{i=1}^{N}\left[
m_{i}^{\left(  k\right)  2}-E\operatorname*{Th}\nolimits^{2}\left(
\sum\nolimits_{s=1}^{k-2}\gamma_{s}\zeta_{i}^{\left(  s\right)  }%
+\sqrt{q-\Gamma_{k-2}^{2}}Z_{k-1}\right)  \right]  \right\vert \geq t\right)
\\
\leq C\exp\left[  -Ct^{2}N\right]  ,
\end{multline*}
where $C$ depends on $k,\beta,h,$ but is non-random, as the bound in Lemma
\ref{Le_chi} depends only on the the Lipshitz constant, and the other parameters.

We proceed in this way, replacing $\zeta_{i}^{\left(  s\right)  },\ s\leq k-2$
successively by $Z_{k-2},Z_{k-3},\ldots,Z_{1},$ condition first on
$\mathcal{G}_{k-3},$ etc. This finally leads to%
\[
\left\Vert \mathbf{m}^{\left(  k\right)  }\right\Vert ^{2}\simeq
E\operatorname*{Th}\nolimits^{2}\left(  \sum\nolimits_{s=1}^{k-2}\gamma
_{s}Z_{s}+\sqrt{q-\Gamma_{k-2}^{2}}Z_{k-1}\right)  =q.
\]

(\ref{ipr_mm2}) follows by a straightforward modification: The case $j=1$ is
trivial, and so we assume $j\geq2.$ As $j<k,$ the conditioning on
$\mathcal{G}_{k-2}$ fixes $\mathbf{m}^{\left(  j\right)  },$ and we therefore
get in the first step%
\begin{multline*}
\mathbb{P}_{k-2}\left(  \left\vert \frac{1}{N}\sum_{i=1}^{N}\left[
m_{i}^{\left(  j\right)  }m_{i}^{\left(  k\right)  }-m_{i}^{\left(  j\right)
}E\operatorname*{Th}\left(  \sum_{s=1}^{k-2}\gamma_{s}\zeta_{i}^{\left(
s\right)  }+\sqrt{q-\Gamma_{k-2}^{2}}Z_{k-1}\right)  \right]  \right\vert \geq
t\right) \\
\leq C\exp\left[  -Ct^{2}N\right]  .
\end{multline*}
This replacement, we do up to replacing $\mathbf{\zeta}^{\left(  j\right)  }$
which is $\mathcal{G}_{j}$-m.b. whereas $\mathbf{m}^{\left(  j\right)  }$ is
$\mathcal{G}_{j-1}$-m.b. We therefore obtain%
\begin{align*}
\frac{1}{N}\sum_{i=1}^{N}m_{i}^{\left(  j\right)  }m_{i}^{\left(  k\right)  }
& \simeq\frac{1}{N}\sum_{i=1}^{N}m_{i}^{\left(  j\right)  }E\operatorname*{Th}%
\left(  \sum_{s=1}^{j-1}\gamma_{s}\zeta_{i}^{\left(  s\right)  }+\sum
_{s=j}^{k-2}\gamma_{s}Z_{s}+\sqrt{q-\Gamma_{k-2}^{2}}Z_{k-1}\right) \\
& =\frac{1}{N}\sum_{i=1}^{N}m_{i}^{\left(  j\right)  }E\operatorname*{Th}%
\left(  \sum_{s=1}^{j-1}\gamma_{s}\zeta_{i}^{\left(  s\right)  }%
+\sqrt{q-\Gamma_{j-1}^{2}}Z_{j}\right)  .
\end{align*}
Performing this conditioning argument now with respect to $\mathcal{G}_{j-2},
$ we get first%
\begin{align*}
\frac{1}{N}\sum_{i=1}^{N}m_{i}^{\left(  j\right)  }m_{i}^{\left(  k\right)  }
& \simeq\frac{1}{N}\sum_{i=1}^{N}E\Big [\operatorname*{Th}\left(
\sum\nolimits_{s=1}^{j-2}\gamma_{s}\zeta_{i}^{\left(  s\right)  }%
+\sqrt{q-\Gamma_{j-2}^{2}}Z_{j-1}\right) \\
& \times\operatorname*{Th}\left(  \sum\nolimits_{s=1}^{j-2}\gamma_{s}\zeta
_{i}^{\left(  s\right)  }+\gamma_{j-1}Z_{j-1}+\sqrt{q-\Gamma_{j-1}^{2}}%
Z_{j}\right)  \Big ],
\end{align*}
and now in the same way as for (\ref{ipr_mm1})%
\begin{align*}
\frac{1}{N}\sum_{i=1}^{N}m_{i}^{\left(  j\right)  }m_{i}^{\left(  k\right)  }
& \simeq E\Big [\operatorname*{Th}\left(  \sum\nolimits_{s=1}^{j-2}\gamma
_{s}Z_{s}+\sqrt{q-\Gamma_{j-2}^{2}}Z_{j-1}\right) \\
& \times\operatorname*{Th}\left(  \sum\nolimits_{s=1}^{j-2}\gamma_{s}%
Z_{s}+\gamma_{j-1}Z_{j-1}+\sqrt{q-\Gamma_{j-1}^{2}}Z_{j}\right)  \Big ].
\end{align*}
A simple computation, as in \cite{BoTAP} in the evaluation of (5.12) there,
shows that the right hand side equals $\psi\left(  \rho_{j-1}\right)
=\rho_{j}.$
\end{proof}

\begin{remark}
The argument given here is considerably simpler than the one in \cite{BoTAP}.
On one hand, this is due to the fact that we don't consider here the random
variables given by the iteration (\ref{TAP_iteration}). Also the explicit
representation of the conditional covariances of the $\mathbf{\zeta}$ is very helpful.
\end{remark}

\section{Estimates for the first and second conditional moments}

The two basic results are:

\begin{proposition}
\label{Prop_1st_moment}If $h>0$ and $\beta$ is small enough then%
\begin{equation}
\lim_{k\rightarrow\infty}\limsup_{N\rightarrow\infty}\mathbb{E}\left\vert
\frac{1}{N}\log\mathbb{E}_{k}\left(  Z_{N}\right)  -\operatorname*{RS}\left(
\beta,h\right)  \right\vert =0.\label{first_moment}%
\end{equation}

\end{proposition}

\begin{proposition}
\label{Prop_2nd_moment}Under the same conditions as in Proposition
\ref{Prop_1st_moment},%
\begin{equation}
\lim_{k\rightarrow\infty}\limsup_{N\rightarrow\infty}\mathbb{E}\left\vert
\frac{1}{N}\log\mathbb{E}_{k}\left(  Z_{N}^{2}\right)  -2\operatorname*{RS}%
\left(  \beta,h\right)  \right\vert \leq0.\label{second_moment}%
\end{equation}

\end{proposition}

\begin{remark}
The requirement on $\beta$ is rather unsatisfactory. I believe that at least
Proposition \ref{Prop_1st_moment} is correct in the full AT-region (\ref{AT}).
Actually, only the very last argument given in the proof in the next section
requires an unspecified \textquotedblleft small $\beta$\textquotedblright%
\ argument. The problem is coming from using the Schwarz inequality and the
H\"{o}lder-inequality in the proof, but I haven't found a better estimate.
\end{remark}

The propositions are proved in the next section. We give now the proof of
Theorem \ref{Th_main} based on these propositions.

We will use that, actually for all $\beta,h,$ the free energy is
self-averaging:%
\begin{equation}
\lim_{N\rightarrow\infty}\frac{1}{N}\log Z_{N}=\lim_{N\rightarrow\infty}%
\frac{1}{N}\mathbb{E}\log Z_{N},\label{selfaveraging}%
\end{equation}
assuming the limit on the right hand side exists, which is the result in
\cite{GuTo}. This is a simple consequence of the Gaussian isoperimetric
inequality, a fact which is well known since long. In fact writing
$J_{ij}:=\sqrt{N}g_{ij}$ which are standard Gaussians, we have%
\[
\left\vert \frac{1}{N}\log Z_{N}\left(  J\right)  -\frac{1}{N}\log
Z_{N}\left(  J^{\prime}\right)  \right\vert \leq\frac{\beta}{\sqrt{2N}%
}\left\Vert J-J^{\prime}\right\Vert
\]
where $\left\Vert \cdot\right\Vert $ denotes the Euclidean norm in
$\mathbb{R}^{N\left(  N-1\right)  /2}.$ Therefore%
\[
\mathbb{P}\left(  \left\vert \frac{1}{N}\log Z_{N}-\mathbb{E}\frac{1}{N}\log
Z_{N}\right\vert \geq t\right)  \leq\exp\left[  -t^{2}N/\beta^{2}\right]  .
\]

By Jensen's inequality%
\[
\limsup_{N\rightarrow\infty}\frac{1}{N}\mathbb{E}\log Z_{N}\leq\limsup
_{N\rightarrow\infty}\frac{1}{N}\mathbb{E}\log\mathbb{E}_{k}\left(
Z_{N}\right)
\]
for all $k.$ Therefore, by Proposition \ref{Prop_1st_moment},%
\begin{equation}
\limsup_{N\rightarrow\infty}\frac{1}{N}\mathbb{E}\log Z_{N}\leq
\operatorname*{RS}\left(  \beta,h\right)  .\label{upper_bound}%
\end{equation}

For the estimate in the other direction, we rely on a second moment argument.
For $k,N\in\mathbb{N},$ set $A_{k,N}:=\left\{  Z_{N}\geq\frac{1}{2}%
\mathbb{E}_{k}\left(  Z_{N}\right)  \right\}  $%
\begin{align*}
\mathbb{E}_{k}\left(  Z_{N}\right)   & =\mathbb{E}_{k}\left(  Z_{N}%
;A_{k,N}^{c}\right)  +\mathbb{E}_{k}\left(  Z_{N};A_{k,N}\right) \\
& \leq\frac{1}{2}\mathbb{E}_{k}\left(  Z_{N}\right)  +\sqrt{\mathbb{E}%
_{k}\left(  Z_{N}^{2}\right)  \mathbb{P}_{k}\left(  A_{k,N}\right)  }%
\end{align*}
and therefore%
\begin{equation}
\mathbb{P}_{k}\left(  A_{k,N}\right)  \geq\frac{\mathbb{E}_{k}\left(
Z_{N}\right)  ^{2}}{4\mathbb{E}_{k}\left(  Z_{N}^{2}\right)  }.\label{Palay}%
\end{equation}
Using Proposition \ref{Prop_2nd_moment}, for an arbitrary $\varepsilon>0$
there exists $k_{0}\left(  \varepsilon\right)  $ such that for $k\geq
k_{0}\left(  \varepsilon\right)  $ we find $N_{0}\left(  \varepsilon,k\right)
$ with
\[
\mathbb{P}\left(  \frac{\mathbb{E}_{k}\left(  Z_{N}\right)  ^{2}}%
{4\mathbb{E}_{k}\left(  Z_{N}^{2}\right)  }\geq\mathrm{e}^{-\varepsilon
N}\right)  \geq\frac{1}{2},\ N\geq N_{0}.
\]
and therefore, by (\ref{Palay}), and the definition of $A_{k,N},$%
\[
\mathbb{P}\left(  \mathbb{P}_{k}\left(  \frac{1}{N}\log Z_{N}\geq\frac{1}%
{N}\log\mathbb{E}_{k}\left(  Z_{N}\right)  -\frac{\log2}{N}\right)
\geq\mathrm{e}^{-\varepsilon N}\right)  \geq\frac{1}{2}.
\]

By Proposition \ref{Prop_1st_moment}, we find for any $\varepsilon^{\prime
}>0,$ a $c\left(  \varepsilon^{\prime}\right)  >0$ and a $k_{0}^{\prime
}\left(  \varepsilon^{\prime}\right)  \in\mathbb{N}$ such that for $k\geq
k_{0}\left(  \varepsilon^{\prime}\right)  ,$ we find $N_{0}^{\prime}\left(
\varepsilon^{\prime},k\right)  $ such that for $N\geq N_{0}^{\prime},$ we have%
\[
\mathbb{P}\left(  \frac{1}{N}\log\mathbb{E}_{k}\left(  Z_{N}\right)
\geq\operatorname*{RS}\left(  \beta,h\right)  -\frac{\varepsilon^{\prime}}%
{2}\right)  \geq\frac{3}{4},
\]
and $N^{-1}\log2\leq\varepsilon^{\prime}/2.$ Therefore, for $k\geq\max\left(
k_{0}\left(  \varepsilon\right)  ,k_{0}^{\prime}\left(  \varepsilon^{\prime
}\right)  \right)  ,\ N\geq\max\left(  N_{0}^{\prime},N_{0}\right)  $%
\[
\mathbb{P}\left(  \mathbb{P}_{k}\left(  \frac{1}{N}\log Z_{N}\geq
\operatorname*{RS}\left(  \beta,h\right)  -\varepsilon^{\prime}\right)
\geq\mathrm{e}^{-\varepsilon N}\right)  \geq\frac{1}{4},
\]
implying by the Markov inequality%
\begin{equation}
\mathbb{P}\left(  \frac{1}{N}\log Z_{N}\geq\operatorname*{RS}\left(
\beta,h\right)  -\varepsilon^{\prime}\right)  \geq\frac{1}{4}\mathrm{e}%
^{-\varepsilon N}.\label{Markov}%
\end{equation}

By Gaussian isoperimetry, we have for any $\eta>0$ and large enough $N$%
\[
\mathbb{P}\left(  \left\vert \frac{1}{N}\log Z_{N}-\frac{1}{N}\mathbb{E}\log
Z_{N}\right\vert \leq\eta\right)  \geq1-\exp\left[  -\eta^{2}N/\beta
^{2}\right]  .
\]
If we choose $\varepsilon<\eta^{2}/\beta^{2},$ it follows that for $N$ large
enough one has%
\[
\frac{1}{N}\mathbb{E}\log Z_{N}\geq\operatorname*{RS}\left(  \beta,h\right)
-\varepsilon^{\prime}-\eta
\]
and as $\eta$ and $\varepsilon^{\prime}$ are arbitrary, we get%
\[
\liminf_{N\rightarrow\infty}\frac{1}{N}\mathbb{E}\log Z_{N}\geq
\operatorname*{RS}\left(  \beta,h\right)  .
\]
Together with (\ref{upper_bound}), this proves%
\[
\liminf_{N\rightarrow\infty}\frac{1}{N}\mathbb{E}\log Z_{N}=\operatorname*{RS}%
\left(  \beta,h\right)  .
\]

\section{Proofs of the propositions}

\begin{proof}
[Proof of Proposition \ref{Prop_1st_moment}]%
\begin{align*}
\mathbb{E}_{k}\left(  Z_{N}\right)   & =\sum_{\mathbf{\sigma}}2^{-N}%
\exp\left[  h\sum\nolimits_{i}\sigma_{i}\right]  \mathbb{E}_{k}\left(
\exp\left[  \frac{\beta N}{\sqrt{2}}\left\langle \mathbf{g\sigma
},\mathbf{\sigma}\right\rangle \right]  \right) \\
& =\sum_{\mathbf{\sigma}}2^{-N}\exp\left[  h\sum\nolimits_{i}\sigma_{i}%
+\frac{\beta N}{\sqrt{2}}\sum_{s=1}^{k}\left\langle \mathbf{\rho}^{\left(
s\right)  }\mathbf{\sigma},\mathbf{\sigma}\right\rangle \right] \\
& \times\mathbb{E}_{k}\left(  \exp\left[  \frac{\beta N}{\sqrt{2}}\left\langle
\mathbf{g}^{\left(  k+1\right)  }\mathbf{\sigma},\mathbf{\sigma}\right\rangle
\right]  \right)  .
\end{align*}

$\mathbf{g}^{\left(  k+1\right)  }$ is Gaussian conditionally on
$\mathcal{G}_{k},$ and therefore%
\[
\mathbb{E}_{k}\left(  \exp\left[  \frac{\beta N}{\sqrt{2}}\left\langle
\mathbf{g}^{\left(  k+1\right)  }\mathbf{\sigma},\mathbf{\sigma}\right\rangle
\right]  \right)  =\exp\left[  \frac{\beta^{2}N^{2}}{4}\mathbb{E}%
_{k}\left\langle \mathbf{g}^{\left(  k+1\right)  }\mathbf{\sigma
},\mathbf{\sigma}\right\rangle ^{2}\right]
\]
According to Proposition \ref{Prop_main_gk} a), b)%
\[
\mathbb{E}_{k}\left\langle \mathbf{g}^{\left(  k+1\right)  }\mathbf{\sigma
},\mathbf{\sigma}\right\rangle ^{2}=\frac{1}{N}\left(  1-\sum\nolimits_{r=1}%
^{k}\left\langle \mathbf{\phi}^{\left(  r\right)  },\mathbf{\sigma
}\right\rangle ^{2}\right)  ^{2}.
\]
Therefore%
\begin{align*}
\mathbb{E}_{k}\left(  Z_{N}\right)   & =\sum_{\mathbf{\sigma}}2^{-N}%
\exp\Big [h\sum\nolimits_{i}\sigma_{i}+\frac{\beta N}{\sqrt{2}}\sum_{s=1}%
^{k}\left\langle \mathbf{\rho}^{\left(  s\right)  }\mathbf{\sigma
},\mathbf{\sigma}\right\rangle \\
& +\frac{\beta^{2}N}{4}\left(  1-\sum\nolimits_{r=1}^{k}\left\langle
\mathbf{\phi}^{\left(  r\right)  },\mathbf{\sigma}\right\rangle ^{2}\right)
^{2}\Big ].
\end{align*}

With $\mathbf{h}^{\left(  k+1\right)  }$ and $\mathbf{m}^{\left(  k+1\right)
}$ defined in (\ref{hk_Def}), (\ref{mk_construktion}), which are
$\mathcal{F}_{k}$-m.b., we put%
\[
p^{\left(  k\right)  }\left(  \mathbf{\sigma}\right)  :=2^{-N}\frac
{\exp\left[  N\left\langle \mathbf{h}^{\left(  k+1\right)  },\mathbf{\sigma
}\right\rangle \right]  }{\prod_{i=1}^{N}\cosh\left(  h_{i}^{\left(
k+1\right)  }\right)  },
\]
which is the product measure of tilted coin tossing, the $\sigma_{i}$ having
mean $m_{i}^{\left(  k+1\right)  }.$ Then,%
\begin{equation}
\mathbb{E}_{k}\left(  Z_{N}\right)  =\exp\left[  \sum\nolimits_{i=1}^{N}%
\log\cosh\left(  h_{i}^{\left(  k+1\right)  }\right)  \right]  \sum
_{\mathbf{\sigma}}p^{\left(  k\right)  }\left(  \mathbf{\sigma}\right)
\exp\left[  N\beta F_{N,k}\left(  \sigma\right)  \right]  ,\label{Ek_ZN_exact}%
\end{equation}
where with $\gamma_{s}$%
\begin{align}
F_{N,k}\left(  \mathbf{\sigma}\right)   & :=\sum_{s=1}^{k}\left\langle
2^{-1/2}\mathbf{\rho}^{\left(  s\right)  }\mathbf{\sigma},\mathbf{\sigma
}\right\rangle -\sum_{s=1}^{k-1}\gamma_{s}\left\langle \mathbf{\zeta}^{\left(
s\right)  },\mathbf{\sigma}\right\rangle \label{FNk_Def}\\
& -\sqrt{q-\Gamma_{k-1}^{2}}\left\langle \mathbf{\zeta}^{\left(  k\right)
},\mathbf{\sigma}\right\rangle +\frac{\beta}{4}\left(  1-\sum\nolimits_{r=1}%
^{k}\left\langle \mathbf{\phi}^{\left(  r\right)  },\mathbf{\sigma
}\right\rangle ^{2}\right)  ^{2}.\nonumber
\end{align}
Up to here, this is an exact computation.

The first part on the right hand side of (\ref{Ek_ZN_exact}) does not depend
on $\mathbf{\sigma}$, and by Lemma \ref{Le_means_h}, we get for any $k:$%
\[
\lim_{N\rightarrow\infty}\mathbb{E}\left\vert \frac{1}{N}\sum\nolimits_{i=1}%
^{N}\log\cosh\left(  h_{i}^{\left(  k+1\right)  }\right)  -E\log\cosh\left(
h+\beta\sqrt{q}Z\right)  \right\vert =0
\]
and therefore, we only have to prove that with%
\[
Z\left(  F_{N,k}\right)  :=\sum_{\mathbf{\sigma}}p^{\left(  k\right)  }\left(
\mathbf{\sigma}\right)  \exp\left[  N\beta F_{N,k}\left(  \mathbf{\sigma
}\right)  \right]
\]
we have%
\begin{equation}
\lim_{k\rightarrow\infty}\lim_{N\rightarrow\infty}\mathbb{E}\left\vert
\frac{1}{N}\log Z\left(  F_{N,k}\right)  -\frac{\beta^{2}\left(  1-q\right)
}{4}\right\vert =0.\label{main_task}%
\end{equation}

We will perform a number of approximations which are negligible in the
$N\rightarrow\infty,\ k\rightarrow\infty,$ in this order. More precisely,
consider a random function%
\[
F_{N,k}^{^{\prime}}\left(  \mathbf{\sigma}\right)  =F_{N,k}\left(
\mathbf{\sigma}\right)  +\Delta_{N,k}\left(  \mathbf{\sigma}\right)
\]
with the property that%
\begin{equation}
\lim_{k\rightarrow\infty}\limsup_{N\rightarrow\infty}\mathbb{E}\sup
\nolimits_{\mathbf{\sigma}}\left\vert \Delta_{N,k}\left(  \mathbf{\sigma
}\right)  \right\vert =0,\label{est_Delta}%
\end{equation}
then%
\begin{equation}
\lim_{k\rightarrow\infty}\lim_{N\rightarrow\infty}\mathbb{E}\left\vert
\frac{1}{N}\log Z\left(  F_{N,k}\right)  -\frac{1}{N}\log Z\left(
F_{N,k}^{\prime}\right)  \right\vert =0.\label{F_by_Fprime}%
\end{equation}

For instance, taking $\Delta_{N,k}\left(  \mathbf{\sigma}\right)  :=\gamma
_{k}\left\langle \mathbf{\zeta}^{\left(  k\right)  },\mathbf{\sigma
}\right\rangle ,$ we have $\sup\nolimits_{\mathbf{\sigma}}\left\vert
\Delta_{N,k}\left(  \mathbf{\sigma}\right)  \right\vert \leq\gamma
_{k}\left\Vert \mathbf{\zeta}^{\left(  k\right)  }\right\Vert ,$ and using the
covariance structure of $\mathbf{\zeta}^{\left(  k\right)  }$ in Proposition
\ref{Prop_main_gk} c), we have $\sup_{k}\mathbb{E}\left\Vert \mathbf{\zeta
}^{\left(  k\right)  }\right\Vert \leq1.$ As $\gamma_{k}\rightarrow0$ for
$k\rightarrow\infty,$ (\ref{est_Delta}) is satisfied. By the same reasoning,
we can neglect $\sqrt{q-\Gamma_{k-1}^{2}}\left\langle \mathbf{\zeta}^{\left(
k\right)  },\mathbf{\sigma}\right\rangle $ under the AT-condition (\ref{AT}).
Therefore, we can replace $F_{N,k}$ by%
\begin{align*}
F_{N,k}^{\prime}\left(  \mathbf{\sigma}\right)   & :=\sum_{s=1}^{k}%
\left\langle 2^{-1/2}\mathbf{\rho}^{\left(  s\right)  }\mathbf{\sigma
},\mathbf{\sigma}\right\rangle -\sum_{s=1}^{k}\gamma_{s}\left\langle
\mathbf{\zeta}^{\left(  s\right)  },\mathbf{\sigma}\right\rangle \\
& +\frac{\beta}{4}\left(  1-\sum\nolimits_{r=1}^{k}\left\langle \mathbf{\phi
}^{\left(  r\right)  },\mathbf{\sigma}\right\rangle ^{2}\right)  ^{2},
\end{align*}
and get (\ref{F_by_Fprime}).

We do a further approximation for the first summand. Plugging in the first two
summands of the definition of $\mathbf{\rho}^{\left(  s\right)  }$
(\ref{rho}), the contribution to $\left\langle 2^{-1/2}\mathbf{\rho}^{\left(
s\right)  }\mathbf{\sigma},\mathbf{\sigma}\right\rangle $ is exactly
$\left\langle \mathbf{\phi}^{\left(  s\right)  },\mathbf{\sigma}\right\rangle
\left\langle \mathbf{\zeta}^{\left(  s\right)  },\mathbf{\sigma}\right\rangle
.$ The third term gives $\left\langle \mathbf{\phi}^{\left(  s\right)
},\mathbf{\xi}^{\left(  s\right)  }\right\rangle \left\langle \mathbf{\phi
}^{\left(  s\right)  },\mathbf{\sigma}\right\rangle ^{2},$ and we claim that
we can neglect that. Indeed%
\[
\sup_{\mathbf{\sigma}}\left\vert \left\langle \mathbf{\phi}^{\left(  s\right)
},\mathbf{\xi}^{\left(  s\right)  }\right\rangle \left\langle \mathbf{\phi
}^{\left(  s\right)  },\mathbf{\sigma}\right\rangle ^{2}\right\vert
\leq\left\vert \left\langle \mathbf{\phi}^{\left(  s\right)  },\mathbf{\xi
}^{\left(  s\right)  }\right\rangle \right\vert ,
\]
and using Lemma \ref{Le_phi_xi_corr}, we see that%
\[
\limsup_{N\rightarrow\infty}\mathbb{E}\sup\nolimits_{\mathbf{\sigma}%
}\left\vert \left\langle \mathbf{\phi}^{\left(  s\right)  },\mathbf{\xi
}^{\left(  s\right)  }\right\rangle \left\langle \mathbf{\phi}^{\left(
s\right)  },\mathbf{\sigma}\right\rangle ^{2}\right\vert =0
\]
for all $s.$ Therefore, we can indeed neglect this part. We now center the
$\mathbf{\sigma}$ by putting%
\[
\mathbf{\hat{\sigma}}^{\left(  k\right)  }:=\mathbf{\sigma-\mathbf{m}%
}^{\left(  k+1\right)  }.
\]
Then%
\begin{multline*}
\sum_{s=1}^{k}\left\langle \mathbf{\phi}^{\left(  s\right)  },\mathbf{\sigma
}\right\rangle \left\langle \mathbf{\zeta}^{\left(  s\right)  },\mathbf{\sigma
}\right\rangle =\sum_{s=1}^{k}\left\langle \mathbf{\phi}^{\left(  s\right)
},\mathbf{\hat{\sigma}}^{\left(  k\right)  }+\mathbf{\mathbf{m}}^{\left(
k+1\right)  }\right\rangle \left\langle \mathbf{\zeta}^{\left(  s\right)
},\mathbf{\hat{\sigma}}^{\left(  k\right)  }+\mathbf{\mathbf{m}}^{\left(
k+1\right)  }\right\rangle \\
=\sum_{s=1}^{k}\left\langle \mathbf{\phi}^{\left(  s\right)  },\mathbf{\hat
{\sigma}}^{\left(  k\right)  }\right\rangle \left\langle \mathbf{\zeta
}^{\left(  s\right)  },\mathbf{\hat{\sigma}}^{\left(  k\right)  }\right\rangle
+\sum_{s=1}^{k}\left\langle \mathbf{\phi}^{\left(  s\right)  }%
,\mathbf{\mathbf{m}}^{\left(  k+1\right)  }\right\rangle \left\langle
\mathbf{\zeta}^{\left(  s\right)  },\mathbf{\hat{\sigma}}^{\left(  k\right)
}\right\rangle \\
+\sum_{s=1}^{k}\left\langle \mathbf{\phi}^{\left(  s\right)  },\mathbf{\hat
{\sigma}}^{\left(  k\right)  }\right\rangle \left\langle \mathbf{\zeta
}^{\left(  s\right)  },\mathbf{\mathbf{m}}^{\left(  k+1\right)  }\right\rangle
+\sum_{s=1}^{k}\left\langle \mathbf{\phi}^{\left(  s\right)  }%
,\mathbf{\mathbf{m}}^{\left(  k+1\right)  }\right\rangle \left\langle
\mathbf{\zeta}^{\left(  s\right)  },\mathbf{\mathbf{m}}^{\left(  k+1\right)
}\right\rangle .
\end{multline*}
We claim that we can replace the second summand on the right hand side by
$\sum_{s=1}^{k}\gamma_{s}\left\langle \mathbf{\zeta}^{\left(  s\right)
},\mathbf{\hat{\sigma}}^{\left(  k\right)  }\right\rangle .$ Indeed%
\[
\left\vert \left\langle \mathbf{\zeta}^{\left(  s\right)  },\mathbf{\hat
{\sigma}}^{\left(  k\right)  }\right\rangle \left[  \left\langle \mathbf{\phi
}^{\left(  s\right)  },\mathbf{m}^{\left(  k+1\right)  }\right\rangle
-\gamma_{s}\right]  \right\vert \leq\left\Vert \mathbf{\zeta}^{\left(
s\right)  }\right\Vert \left\Vert \left\langle \mathbf{\phi}^{\left(
s\right)  },\mathbf{m}^{\left(  k+1\right)  }\right\rangle -\gamma
_{s}\right\Vert ,
\]
and%
\begin{align*}
\mathbb{E}\left(  \left\Vert \mathbf{\zeta}^{\left(  s\right)  }\right\Vert
\left\Vert \left\langle \mathbf{\phi}^{\left(  s\right)  },\mathbf{m}^{\left(
k+1\right)  }\right\rangle -\gamma_{s}\right\Vert \right)   & \leq
\sqrt{\mathbb{E}\left\Vert \mathbf{\zeta}^{\left(  s\right)  }\right\Vert
^{2}\mathbb{E}\left\Vert \left\langle \mathbf{\phi}^{\left(  s\right)
},\mathbf{m}^{\left(  k+1\right)  }\right\rangle -\gamma_{s}\right\Vert ^{2}%
}\\
& \leq\operatorname*{const}\times\sqrt{\mathbb{E}\left\Vert \left\langle
\mathbf{\phi}^{\left(  s\right)  },\mathbf{m}^{\left(  k+1\right)
}\right\rangle -\gamma_{s}\right\Vert ^{2}}%
\end{align*}
which converges to $0$ for $N\rightarrow\infty,$ by Proposition
\ref{Prop_inner_products} a). In a similar way, using Lemma
\ref{Le_gamma_approx}, we can replace%
\[
\sum_{s=1}^{k}\left\langle \mathbf{\phi}^{\left(  s\right)  },\mathbf{\hat
{\sigma}}^{\left(  k\right)  }\right\rangle \left\langle \mathbf{\zeta
}^{\left(  s\right)  },\mathbf{\mathbf{m}}^{\left(  k+1\right)  }%
\right\rangle
\]
by%
\[
\beta\left(  1-q\right)  \sum_{s=1}^{k}\gamma_{s}\left\langle \mathbf{\phi
}^{\left(  s\right)  },\mathbf{\hat{\sigma}}^{\left(  k\right)  }\right\rangle
.
\]
In the end, we replace $F_{N,k}^{\prime}$ by%
\begin{align*}
F_{N,k}^{\prime\prime}\left(  \mathbf{\sigma}\right)   & :=\sum_{s=1}%
^{k}\left\langle \mathbf{\phi}^{\left(  s\right)  },\mathbf{\hat{\sigma}%
}^{\left(  k\right)  }\right\rangle \left\langle \mathbf{\zeta}^{\left(
s\right)  },\mathbf{\hat{\sigma}}^{\left(  k\right)  }\right\rangle
+\beta\left(  1-q\right)  \sum_{s=1}^{k}\gamma_{s}\left\langle \mathbf{\phi
}^{\left(  s\right)  },\mathbf{\hat{\sigma}}^{\left(  k\right)  }\right\rangle
\\
& +\frac{\beta}{4}\left(  1-\sum\nolimits_{r=1}^{k}\left\langle \mathbf{\phi
}^{\left(  r\right)  },\mathbf{\sigma}\right\rangle ^{2}\right)  ^{2},
\end{align*}
achieving%
\begin{equation}
\lim_{k\rightarrow\infty}\lim_{N\rightarrow\infty}\mathbb{E}\left\vert
\frac{1}{N}\log Z\left(  F_{N,k}^{\prime}\right)  -\frac{1}{N}\log Z\left(
F_{N,k}^{\prime\prime}\right)  \right\vert =0.\label{Fprime_by_F2prime}%
\end{equation}
where we have made repeated use of Proposition \ref{Prop_inner_products} and
Lemma \ref{Le_gamma_approx}, and $\sum_{s=1}^{k}\gamma_{s}^{2}\rightarrow q, $
as $k\rightarrow\infty,$ under the AT-condition. Using (\ref{F_by_Fprime}), it
therefore remains to prove%
\[
\lim_{k\rightarrow\infty}\lim_{N\rightarrow\infty}\mathbb{E}\left\vert
\frac{1}{N}\log Z\left(  F_{N,k}^{\prime\prime}\right)  -\frac{\beta
^{2}\left(  1-q\right)  }{4}\right\vert =0.
\]

The most \textquotedblleft dangerous\textquotedblright\ part in In
$F_{N,k}^{\prime\prime}$ is the presence of $\sum_{s=1}^{k}\gamma
_{s}\left\langle \mathbf{\phi}^{\left(  s\right)  },\mathbf{\hat{\sigma}%
}^{\left(  k\right)  }\right\rangle $, but fortunately, it cancels in leading
order when centering the third part.%
\[
\sum\nolimits_{r=1}^{k}\left\langle \mathbf{\phi}^{\left(  r\right)
},\mathbf{\sigma}\right\rangle ^{2}=\sum\nolimits_{r=1}^{k}\left[
\left\langle \mathbf{\phi}^{\left(  r\right)  },\mathbf{\hat{\sigma}}^{\left(
k\right)  }\right\rangle +\left\langle \mathbf{\phi}^{\left(  r\right)
},\mathbf{m}^{\left(  k+1\right)  }\right\rangle \right]  ^{2}.
\]
For the same reason as repeatedly use above, we may replace $\left\langle
\mathbf{\phi}^{\left(  r\right)  },\mathbf{m}^{\left(  k+1\right)
}\right\rangle $ by $\gamma_{r}$ (in the $N\rightarrow\infty,\ k\rightarrow
\infty$ limit), and replace $\sum_{r=1}^{k}\gamma_{r}^{2}$ by $q$ under the
AT-condition. By these approximations, we replace the right hand side of the
expression above by%
\[
\sum\nolimits_{r=1}^{k}\left[  \left\langle \mathbf{\phi}^{\left(  r\right)
},\mathbf{\hat{\sigma}}^{\left(  k\right)  }\right\rangle +\gamma_{r}\right]
^{2}\approx q+2Y_{k}+S_{k}^{2},
\]
where%
\begin{align*}
Y_{k}  & :=\sum\nolimits_{r=1}^{k}\gamma_{r}\left\langle \mathbf{\phi
}^{\left(  r\right)  },\mathbf{\hat{\sigma}}^{\left(  k\right)  }\right\rangle
,\\
S_{k}^{2}  & :=\sum\nolimits_{r=1}^{k}\left\langle \mathbf{\phi}^{\left(
r\right)  },\mathbf{\hat{\sigma}}^{\left(  k\right)  }\right\rangle ^{2}.
\end{align*}
Therefore, with these approximations, we have%
\begin{align*}
& \beta\left(  1-q\right)  Y_{k}+\frac{\beta}{4}\left(  1-\sum\nolimits_{r=1}%
^{k}\left\langle \mathbf{\phi}^{\left(  r\right)  },\mathbf{\sigma
}\right\rangle ^{2}\right)  ^{2}\\
& \approx\beta\left(  1-q\right)  Y_{k}+\frac{\beta}{4}\left(  1-q-2Y_{k}%
-S_{k}^{2}\right)  ^{2}\\
& =\frac{\beta\left(  1-q\right)  ^{2}}{4}+\beta Y_{k}^{2}+\frac{\beta}%
{4}S_{k}^{4}-\frac{\beta}{2}\left(  1-q\right)  S_{k}^{2}+\beta Y_{k}S_{k}%
^{2}.
\end{align*}
The first summand is exactly what we want, and we \textquotedblleft
only\textquotedblright\ have to check that the rest does not harm. In other
words, putting%
\[
F_{N,k}^{\prime\prime\prime}\left(  \mathbf{\sigma}\right)  :=\sum_{s=1}%
^{k}\left\langle \mathbf{\phi}^{\left(  s\right)  },\mathbf{\hat{\sigma}%
}^{\left(  k\right)  }\right\rangle \left\langle \mathbf{\zeta}^{\left(
s\right)  },\mathbf{\hat{\sigma}}^{\left(  k\right)  }\right\rangle +\beta
Y_{k}^{2}+\frac{\beta}{4}S_{k}^{4}-\frac{\beta}{2}\left(  1-q\right)
S_{k}^{2}+\beta Y_{k}S_{k}^{2},
\]
we have%
\begin{equation}
\lim_{k\rightarrow\infty}\lim_{N\rightarrow\infty}\mathbb{E}\left\vert
\frac{1}{N}\log Z\left(  F_{N,k}^{\prime\prime}\right)  -\frac{1}{N}\log
Z\left(  \frac{\beta\left(  1-q\right)  ^{2}}{4}+F_{N,k}^{\prime\prime\prime
}\right)  \right\vert =0,\label{F2prime_by_F3prime}%
\end{equation}
and using (\ref{F_by_Fprime}), (\ref{Fprime_by_F2prime}), and
(\ref{F2prime_by_F3prime}), it remains to prove%
\begin{equation}
\lim_{k\rightarrow\infty}\lim_{N\rightarrow\infty}\mathbb{E}\left\vert
\frac{1}{N}\log Z\left(  F_{N,k}^{\prime\prime\prime}\right)  \right\vert
=0.\label{CurieWeiss}%
\end{equation}

This is a somewhat complicated Curie-Weiss type computation. An important
point is that $F_{N,k}^{\prime\prime\prime}$ contains only summands which are
at least quadratic in the $\mathbf{\hat{\sigma}}^{\left(  k\right)  }.$ If
there would be a linear term, (\ref{CurieWeiss}) would for any $\beta>0$ not
be true, of course. I strongly believe that (\ref{CurieWeiss}) is correct
under the AT-condition (\ref{AT}), but a prove eludes me. The reader should
also be aware, that we haven't lost anything in the AT-region. In other words,
if for a parameter $\left(  \beta,h\right)  $ satisfying (\ref{AT}),
(\ref{CurieWeiss}) is not true, then for these $\left(  \beta,h\right)  ,$
(\ref{first_moment}) is not correct.

First remark that%
\[
\frac{1}{N}\log Z\left(  F_{N,k}^{\prime\prime\prime}\right)  \geq\beta
\sum\nolimits_{\mathbf{\sigma}}p^{\left(  k\right)  }\left(  \mathbf{\sigma
}\right)  F_{N,k}^{\prime\prime\prime}\left(  \mathbf{\sigma}\right)
\]
and $\mathbb{E}\left\vert \sum\nolimits_{\mathbf{\sigma}}p^{\left(  k\right)
}\left(  \mathbf{\sigma}\right)  F_{N,k}^{\prime\prime\prime}\left(
\mathbf{\sigma}\right)  \right\vert =O\left(  N^{-1/2}\right)  $ through the
independence of the components under $p^{\left(  k\right)  }\left(
\mathbf{\sigma}\right)  $ and the centering.

It remains to prove the upper bound. We use some rather crude and certainly
not optimal bounds.%
\begin{align*}
\sum_{s=1}^{k}\left\langle \mathbf{\phi}^{\left(  s\right)  },\mathbf{\hat
{\sigma}}^{\left(  k\right)  }\right\rangle \left\langle \mathbf{\zeta
}^{\left(  s\right)  },\mathbf{\hat{\sigma}}^{\left(  k\right)  }%
\right\rangle  & \leq\sqrt{S_{k}^{2}\sum\nolimits_{s=1}^{k}\left\langle
\mathbf{\zeta}^{\left(  s\right)  },\mathbf{\hat{\sigma}}^{\left(  k\right)
}\right\rangle ^{2}}\\
& \leq\frac{1}{2}S_{k}^{2}+\frac{1}{2}\sum\nolimits_{s=1}^{k}\left\langle
\mathbf{\zeta}^{\left(  s\right)  },\mathbf{\hat{\sigma}}^{\left(  k\right)
}\right\rangle ^{2}.
\end{align*}
Also%
\[
\left\vert Y_{k}\right\vert \leq\sum_{s=1}^{k}\gamma_{s}^{2}\left\Vert
\mathbf{\hat{\sigma}}^{\left(  k\right)  }\right\Vert \leq q\left\Vert
\mathbf{\hat{\sigma}}^{\left(  k\right)  }\right\Vert \leq2q,
\]%
\[
S_{k}^{2}\leq\left\Vert \mathbf{\hat{\sigma}}^{\left(  k\right)  }\right\Vert
\leq2.
\]
Using these crude estimates, and the H\"{o}lder inequality, one sees that it
satisfies to prove%
\begin{align}
\limsup_{N\rightarrow\infty}\frac{1}{N}\mathbb{E}\log\sum_{\mathbf{\sigma}%
}p\left(  \mathbf{\sigma}\right)  \exp\left[  \lambda NS_{k}^{2}\right]    &
\leq0,\label{Est_S}\\
\limsup_{N\rightarrow\infty}\frac{1}{N}\mathbb{E}\log\sum_{\mathbf{\sigma}%
}p\left(  \mathbf{\sigma}\right)  \exp\left[  \lambda NY_{k}^{2}\right]    &
\leq0,\label{Est_Z}\\
\limsup_{N\rightarrow\infty}\frac{1}{N}\mathbb{E}\log\sum_{\mathbf{\sigma}%
}p^{\left(  k\right)  }\left(  \mathbf{\sigma}\right)  \exp\left[  \lambda
N\sum\nolimits_{s=1}^{k}\left\langle \mathbf{\zeta}^{\left(  s\right)
},\mathbf{\hat{\sigma}}\right\rangle ^{2}\right]    & \leq0,\label{Est_zeta}%
\end{align}
for small enough $\lambda>0,$ where \textquotedblleft small
enough\textquotedblright\ \textit{does not depend} on $k.$ This latter
requirement looks somewhat dangerous, but here it helps that the
$\phi^{\left(  s\right)  }$ are orthogonal with respect to out inner product
on $\mathbb{R}^{N},$ and the $\mathbf{\zeta}^{\left(  s\right)  }$ are
approximately so. We start with (\ref{Est_S})%
\begin{align*}
\sum_{\mathbf{\sigma}}p\left(  \mathbf{\sigma}\right)  \exp\left[  \lambda
NS_{k}^{2}\right]    & =\sum_{\mathbf{\sigma}}p^{\left(  k\right)  }\left(
\mathbf{\sigma}\right)  \exp\left[  \lambda N\sum\nolimits_{s=1}%
^{k}\left\langle \mathbf{\phi}^{\left(  s\right)  },\mathbf{\hat{\sigma}%
}^{\left(  k\right)  }\right\rangle ^{2}\right]  \\
& =E\sum_{\mathbf{\sigma}}p^{\left(  k\right)  }\left(  \mathbf{\sigma
}\right)  \exp\left[  \sum_{i=1}^{N}\left(  \sum\nolimits_{s=1}^{k}Z_{s}%
\sqrt{\frac{2\lambda}{N}}\phi_{i}^{\left(  s\right)  }\right)  \hat{\sigma
}_{i}\right]  \\
& \leq E\exp\left[  \sum_{i=1}^{N}\chi_{i}\left(  \sum\nolimits_{s=1}^{k}%
Z_{s}\sqrt{\frac{2\lambda}{N}}\phi_{i}^{\left(  s\right)  }\right)  \right]  .
\end{align*}
where%
\[
\chi_{i}\left(  x\right)  :=\log\cosh\left(  h_{i}+x\right)  -\log\cosh\left(
h_{i}\right)  -xm_{i}.
\]
By Lemma \ref{Le_chi}, we have $\chi_{i}\left(  x\right)  \leq x^{2}/2,$ so,
using also the fact that the $\mathbf{\phi}^{\left(  s\right)  }$ are
orthonormal, one has that the above is%
\[
\leq E\exp\left[  \frac{\lambda}{N}\sum_{i=1}^{N}\left(  \sum\nolimits_{s=1}%
^{k}Z_{s}\phi_{i}^{\left(  s\right)  }\right)  ^{2}\right]  =\left(
E\exp\left[  \lambda Z\right]  \right)  ^{k}%
\]
which finite for $\lambda<1/2.$ Therefore, we have for this part a
deterministic upper bound and therefore (\ref{Est_S}) follows.

We next prove (\ref{Est_Z}).%
\begin{align*}
\sum_{\mathbf{\sigma}}p\left(  \mathbf{\sigma}\right)  \exp\left[  \lambda
NY_{k}^{2}\right]    & =E\sum_{\mathbf{\sigma}}p^{\left(  k\right)  }\left(
\mathbf{\sigma}\right)  \exp\left[  \sqrt{\frac{2\lambda}{N}}Z\sum
\nolimits_{r=1}^{k}\gamma_{r}\sum\nolimits_{i=1}^{N}\phi_{i}^{\left(
r\right)  }\sigma_{i}^{\left(  k+1\right)  }\right]  \\
& =E\sum_{\mathbf{\sigma}}p^{\left(  k\right)  }\left(  \mathbf{\sigma
}\right)  \exp\left[  \sum\nolimits_{i=1}^{N}\sigma_{i}^{\left(  k+1\right)
}\sqrt{\frac{2\lambda}{N}}Z\sum\nolimits_{r=1}^{k}\gamma_{r}\phi_{i}^{\left(
r\right)  }\right]  \\
& \leq E\exp\left[  \frac{\lambda}{N}Z^{2}\sum\nolimits_{i=1}^{N}\left(
\sum\nolimits_{r=1}^{k}\gamma_{r}\phi_{i}^{\left(  r\right)  }\right)
^{2}\right]  \\
& =E\exp\left[  \lambda Z^{2}\sum\nolimits_{r=1}^{k}\gamma_{r}^{2}\right]
\leq E\exp\left[  \lambda qZ^{2}\right]  <\infty
\end{align*}
for $\lambda q<1/2.$

(\ref{Est_zeta}) is slightly more complicated. We start in the same way as
above and reach%
\begin{equation}
\sum_{\mathbf{\sigma}}p^{\left(  k\right)  }\left(  \mathbf{\sigma}\right)
\exp\left[  \lambda N\sum_{s=1}^{k}\left\langle \mathbf{\zeta}^{\left(
s\right)  },\mathbf{\hat{\sigma}}\right\rangle ^{2}\right]  =E\exp\left[
\sum_{i=1}^{N}\chi_{i}\left(  \sum\nolimits_{s=1}^{k}Z_{s}\sqrt{\frac
{2\lambda}{N}}\zeta_{i}^{\left(  s\right)  }\right)  \right]
\label{Est_zeta1}%
\end{equation}
Fix and $\varepsilon>0,$ and consider the event%
\[
A_{k,N}:=\bigcup\nolimits_{s:s\leq k}\left\{  \left\Vert \mathbf{\zeta
}^{\left(  s\right)  }\right\Vert ^{2}>1+\varepsilon\right\}  \cup
\bigcup\nolimits_{s.t:s,t\leq k}\left\{  \left\vert \left\langle
\mathbf{\zeta}^{\left(  s\right)  },\mathbf{\zeta}^{\left(  t\right)
}\right\rangle \right\vert >\frac{2\varepsilon}{k}\right\}
\]
On $A_{k,N}^{\mathrm{c}},$ we estimate the rhs of (\ref{Est_zeta1}) by%
\[
\left(  E\exp\left[  \lambda\left(  1+2\varepsilon\right)  Z\right]  \right)
^{k}%
\]
which is finite if $\lambda\left(  1+2\varepsilon\right)  <1/2.$ On the other
hand%
\[
\frac{1}{N}\log\sum_{\mathbf{\sigma}}p\left(  \mathbf{\sigma}\right)
\exp\left[  \lambda N\sum_{s=1}^{k}\left\langle \mathbf{\zeta}^{\left(
s\right)  },\mathbf{\hat{\sigma}}\right\rangle ^{2}\right]  \leq2\lambda
\sum_{s=1}^{k}\left\Vert \mathbf{\zeta}^{\left(  s\right)  }\right\Vert ^{2}%
\]
and by Lemma \ref{Le_zeta}%
\[
\lim_{N\rightarrow\infty}\mathbb{E}\left(  1_{A_{N,k}}\sum\nolimits_{s=1}%
^{k}\left\Vert \mathbf{\zeta}^{\left(  s\right)  }\right\Vert ^{2}\right)
\leq\lim_{N\rightarrow\infty}\sqrt{\mathbb{P}\left(  A_{N,k}\right)  }%
\sqrt{\mathbb{E}\left[  \sum\nolimits_{s=1}^{k}\left\Vert \mathbf{\zeta
}^{\left(  s\right)  }\right\Vert ^{2}\right]  ^{2}}=0
\]
for all $k.$ Therefore,%
\[
\lim_{N\rightarrow\infty}\mathbb{E}\frac{1}{N}\log\sum_{\mathbf{\sigma}%
}p^{\left(  k\right)  }\left(  \mathbf{\sigma}\right)  \exp\left[  \lambda
N\sum_{s=1}^{k}\left\langle \mathbf{\zeta}^{\left(  s\right)  },\mathbf{\hat
{\sigma}}\right\rangle ^{2}\right]  =0
\]

\end{proof}

\begin{proof}
[Proof of Proposition \ref{Prop_2nd_moment}]This is parallel, and we will be
brief. A similar computation as in the previous proof leads to%
\begin{align*}
\mathbb{E}_{k}\left(  Z_{N}^{2}\right)   & =\sum_{\mathbf{\sigma,\tau}}%
2^{-2N}\exp\left[  h\sum\nolimits_{i}\left(  \sigma_{i}+\tau_{i}\right)
+\frac{\beta N}{\sqrt{2}}\sum_{s=1}^{k}\left\langle \mathbf{\rho}^{\left(
s\right)  }\mathbf{\sigma},\mathbf{\sigma}\right\rangle +\left\langle
\mathbf{\rho}^{\left(  s\right)  }\mathbf{\tau},\mathbf{\tau}\right\rangle
\right] \\
& \times\exp\Big [\frac{\beta^{2}N}{4}\Big (\left(  1-\sum\nolimits_{r=1}%
^{k}\left\langle \mathbf{\phi}^{\left(  r\right)  },\mathbf{\sigma
}\right\rangle ^{2}\right)  ^{2}\\
& +\left(  1-\sum\nolimits_{r=1}^{k}\left\langle \mathbf{\phi}^{\left(
r\right)  },\mathbf{\tau}\right\rangle ^{2}\right)  ^{2}\\
& +\left[  \left\langle \mathbf{\sigma},\mathbf{\tau}\right\rangle
-\sum\nolimits_{r=1}^{k}\left\langle \mathbf{\sigma},\mathbf{\phi}^{\left(
r\right)  }\right\rangle \left\langle \mathbf{\tau},\mathbf{\phi}^{\left(
r\right)  }\right\rangle \right]  ^{2}\Big )\Big ].
\end{align*}
The only difference between $\mathbb{E}_{k}\left(  Z_{N}^{2}\right)  $ and
$\left(  \mathbb{E}_{k}Z_{N}\right)  ^{2}$ come from the presence of the last
cross term in the expression above. We therefore only have to check that after
the centering of $\mathbf{\sigma}$ around $\mathbf{m}^{\left(  k+1\right)  },$
and switching to $p^{\left(  k\right)  }\left(  \mathbf{\sigma}\right)
,\ p^{\left(  k\right)  }\left(  \mathbf{\tau}\right)  ,$ this cross term does
not cause problems for $\beta$ small. Writing $\mathbf{\sigma=\hat{\sigma}%
}^{\left(  k\right)  }+\mathbf{m}^{\left(  k+1\right)  }$ and multiplying out,
the only contribution in $\left\langle \mathbf{\sigma},\mathbf{\tau
}\right\rangle -\sum_{r=1}^{k}\left\langle \mathbf{\sigma},\mathbf{\phi
}^{\left(  r\right)  }\right\rangle \left\langle \mathbf{\tau},\mathbf{\phi
}^{\left(  r\right)  }\right\rangle $ which is not linear or quadratic in
$\left(  \mathbf{\hat{\sigma}}^{\left(  k\right)  },\mathbf{\hat{\tau}%
}^{\left(  k\right)  }\right)  $ is%
\[
\left\Vert \mathbf{m}^{\left(  k+1\right)  }\right\Vert ^{2}-\sum
\nolimits_{r=1}^{k}\left\langle \mathbf{m}^{\left(  k+1\right)  }%
,\mathbf{\phi}^{\left(  r\right)  }\right\rangle \left\langle \mathbf{m}%
^{\left(  k+1\right)  },\mathbf{\phi}^{\left(  r\right)  }\right\rangle .
\]
But%
\[
\lim_{k\rightarrow\infty}\lim_{N\rightarrow\infty}\mathbb{E}\left\vert
\left\Vert \mathbf{m}^{\left(  k+1\right)  }\right\Vert ^{2}-\sum
\nolimits_{r=1}^{k}\left\langle \mathbf{m}^{\left(  k+1\right)  }%
,\mathbf{\phi}^{\left(  r\right)  }\right\rangle \left\langle \mathbf{m}%
^{\left(  k+1\right)  },\mathbf{\phi}^{\left(  r\right)  }\right\rangle
\right\vert =0,
\]
so what remains after this (asymptotic) cancellation are terms which are
linear or quadratic in $\left(  \mathbf{\hat{\sigma}}^{\left(  k\right)
},\mathbf{\hat{\tau}}^{\left(  k\right)  }\right)  $. Therefore after squaring
this expression, they are quadratic or of higher order. Writing $G_{N,k}%
\left(  \mathbf{\sigma},\mathbf{\tau}\right)  $ for this, we see in the same
way as in the proof of Proposition \ref{Prop_1st_moment} that%
\begin{align*}
\mathbb{E}_{k}\left(  Z_{N}^{2}\right)   & \leq\exp\left[
2N\operatorname*{RS}\left(  \beta,h\right)  \right]  \sum_{\mathbf{\sigma
,\tau}}p^{\left(  k\right)  }\left(  \mathbf{\sigma}\right)  p^{\left(
k\right)  }\left(  \mathbf{\tau}\right) \\
& \times\exp\left[  N\beta\left(  F_{N,k}^{\prime}\left(  \mathbf{\sigma
}\right)  +F_{N,k}^{\prime}\left(  \mathbf{\tau}\right)  +G_{N,k}\left(
\mathbf{\sigma},\mathbf{\tau}\right)  \right)  \right]  ,
\end{align*}
and with the same argument as before, one sees that for small enough $\beta$%
\begin{multline*}
\lim_{k\rightarrow\infty}\limsup_{N\rightarrow\infty}\frac{1}{N}%
\mathbb{\mathbb{E}}\log\sum_{\mathbf{\sigma,\tau}}p^{\left(  k\right)
}\left(  \mathbf{\sigma}\right)  p^{\left(  k\right)  }\left(  \mathbf{\tau
}\right) \\
\times\exp\left[  N\beta\left(  F_{N,k}^{\prime}\left(  \mathbf{\sigma
}\right)  +F_{N,k}^{\prime}\left(  \mathbf{\tau}\right)  +G_{N,k}\left(
\mathbf{\sigma},\mathbf{\tau}\right)  \right)  \right]  \leq0.
\end{multline*}

\end{proof}

\section{Technical lemmas}

\begin{lemma}
\label{Le_phi_xi_corr}$\left\langle \mathbf{\phi}^{\left(  m\right)
},\mathbf{\xi}^{\left(  m\right)  }\right\rangle $ is (unconditionally)
Gaussian with variance $1/N.$
\end{lemma}

\begin{proof}%
\[
\left\langle \mathbf{\phi}^{\left(  m\right)  },\mathbf{\xi}^{\left(
m\right)  }\right\rangle =\frac{1}{N}\sum_{i,j}\phi_{i}^{\left(  m\right)
}g_{ij}^{\left(  m\right)  }\phi_{j}^{\left(  m\right)  }%
\]
$\mathbf{\phi}^{\left(  m\right)  }$ is $\mathcal{F}_{m-1}$-m.b., and
$\mathbf{g}^{\left(  m\right)  }$ is conditionally Gaussian given
$\mathcal{F}_{m-1}$ with covariances given by (\ref{g_covariances})$.$
Computing the conditional variance, using this expression, yields%
\[
\mathbb{E}_{m-1}\left(  \frac{1}{N}\sum_{i,j}\phi_{i}^{\left(  m\right)
}g_{ij}^{\left(  m\right)  }\phi_{j}^{\left(  m\right)  }\right)  ^{2}%
=\frac{1}{N}.
\]
This proves the claim.
\end{proof}

Below, we denote by $\chi_{n}\left(  x\right)  ,\ x\geq0,$ the density of the
$\chi^{2}$-distribution of degree $n,$ i.e.%
\[
\chi_{n}\left(  x\right)  :=\frac{x^{n/2-1}\mathrm{e}^{-x/2}}{2^{n/2}%
\Gamma\left(  n/2\right)  },
\]
$\Gamma$ here gamma function, and%
\[
\Xi_{n}\left(  x\right)  :=\int_{x}^{\infty}\chi_{n}\left(  y\right)  dy.
\]
For fixed $n$, $\Xi_{n}\left(  x\right)  $ is exponentially decaying for
$x\rightarrow\infty.$

\begin{lemma}
\label{Le_chi}Let $\mathbf{y}^{\left(  1\right)  },\ldots,\mathbf{y}^{\left(
k\right)  }$ be orthonormal vectors in $\mathbb{R}^{N},$ and $\mathbf{X}$ be a
Gaussian random variable with covariances%
\[
\mathbb{E}X_{i}X_{j}=\delta_{ij}-\frac{1}{N}\sum_{s=1}^{k_{1}}y_{i}^{\left(
s\right)  }y_{j}^{\left(  s\right)  }+\frac{1}{N}\sum_{s=k_{1}+1}^{k}%
y_{i}^{\left(  s\right)  }y_{j}^{\left(  s\right)  }%
\]
with $1\leq k_{1}\leq k,$ $k_{2}:=k-k_{1}.$

\begin{enumerate}
\item[a)] If $f_{i}:\mathbb{R\rightarrow R}$ are Lipshitz continuous functions
with%
\[
\lambda:=\sup_{i}\left\Vert f_{i}\right\Vert _{\mathrm{lip}}<\infty
\]
then%
\begin{align*}
& \mathbb{P}\left(  \left\vert \frac{1}{N}\sum\nolimits_{i=1}^{N}\left[
f_{i}\left(  X_{i}\right)  -Ef_{i}\left(  Z\right)  \right]  \right\vert \geq
t\right) \\
& \leq\Xi_{k_{1}}\left(  \frac{t^{2}N}{9k_{1}\lambda^{2}}\right)  +\Xi_{k_{2}%
}\left(  \frac{t^{2}N}{9k_{2}\lambda^{2}}\right)  +\exp\left[  -t^{2}%
N/\lambda^{2}\right]  .
\end{align*}

\item[b)]
\[
\mathbb{E}\left\Vert \mathbf{X}\right\Vert ^{2}=n-k+2k_{2},
\]
and%
\[
\mathbb{P}\left(  \left\Vert \mathbf{X}\right\Vert ^{2}\geq t\right)
=\int_{0}^{tN/2}\Xi_{N-k}\left(  Nt-2x\right)  \chi_{k_{2}}\left(  x\right)
dx
\]

\end{enumerate}
\end{lemma}

\begin{proof}
a) Choose i.i.d. standard Gaussian variables $U_{1},\ldots,U_{N},$ and
$Z_{1},\ldots,Z_{k}.$ Then $\mathbf{Y}$ with%
\[
Y_{i}:=X_{i}+\frac{1}{\sqrt{N}}\sum_{s=1}^{k_{1}}y_{i}^{\left(  s\right)
}Z_{s}%
\]
has the same distribution as $\mathbf{Y}^{\prime}$ given by%
\[
Y_{i}^{\prime}:=U_{i}+\frac{1}{\sqrt{N}}\sum_{s=k_{1}+1}^{k}y_{i}^{\left(
s\right)  }Z_{s}.
\]
Then%
\begin{align*}
& \mathbb{P}\left(  \left\vert \frac{1}{N}\sum\nolimits_{i=1}^{N}\left[
f_{i}\left(  X_{i}\right)  -Ef_{i}\left(  Z\right)  \right]  \right\vert \geq
t\right) \\
& \leq\mathbb{P}\left(  \left\vert \frac{1}{N}\sum\nolimits_{i=1}^{N}\left[
f_{i}\left(  X_{i}\right)  -f_{i}\left(  Y_{i}\right)  \right]  \right\vert
\geq t/3\right) \\
& +\mathbb{P}\left(  \left\vert \frac{1}{N}\sum\nolimits_{i=1}^{N}\left[
f_{i}\left(  Y_{i}^{\prime}\right)  -f_{i}\left(  U_{i}\right)  \right]
\right\vert \geq t/3\right) \\
& +\mathbb{P}\left(  \left\vert \frac{1}{N}\sum\nolimits_{i=1}^{N}\left[
f_{i}\left(  U_{i}\right)  -Ef_{i}\left(  Z\right)  \right]  \right\vert \geq
t/3\right)  .
\end{align*}
Estimating the first two parts using the orthonormality of the $\mathbf{y}%
^{\left(  s\right)  },$ and the last summand using Gaussian isoperimetry,
leads to the desired bound.

b) The covariance matrix $\left(  \mathbb{E}X_{i}X_{j}\right)  _{i,j}$ has
spectrum $\left\{  0,1,2\right\}  $ with multiplicities $k_{1},N-k,k_{2}.$
From that, the estimate follows.
\end{proof}

\begin{lemma}
\label{Le_gm_to_ze}If (\ref{AT}) is satisfied, then $\lim_{k\rightarrow\infty
}\lim_{N\rightarrow\infty}\mathbb{E}\left\Vert \mathbf{g}^{\left(  k\right)
}\mathbf{m}^{\left(  k\right)  }\right\Vert ^{2}=0.$
\end{lemma}

\begin{proof}
As $\mathbf{m}^{\left(  k\right)  }$ is $\mathcal{G}_{k-1}$-m.b.,
$\mathbf{g}^{\left(  k\right)  }\mathbf{m}^{\left(  k\right)  }$ is
conditionally Gaussian with covariances, using Proposition \ref{Prop_main_gk},%
\begin{align}
& \mathbb{E}_{k-1}\left(  \left(  \mathbf{g}^{\left(  k\right)  }%
\mathbf{m}^{\left(  k\right)  }\right)  _{i}\left(  \mathbf{g}^{\left(
k\right)  }\mathbf{m}^{\left(  k\right)  }\right)  _{s}\right)
\label{gkmk_cov}\\
& =\left[  \delta_{is}-\alpha_{is}^{\left(  k-1\right)  }\right]  \left[
\left\Vert \mathbf{m}^{\left(  k\right)  }\right\Vert ^{2}-\sum\nolimits_{m=1}%
^{k-1}\left\langle \mathbf{m}^{\left(  k\right)  },\mathbf{\phi}^{\left(
m\right)  }\right\rangle ^{2}\right]  .\nonumber
\end{align}
Using Lemma \ref{Le_chi}, one gets%
\begin{align*}
\mathbb{E}_{k-1}\left\Vert \mathbf{g}^{\left(  k\right)  }\mathbf{m}^{\left(
k\right)  }\right\Vert ^{2}  & =\left[  \left\Vert \mathbf{m}^{\left(
k\right)  }\right\Vert ^{2}-\sum\nolimits_{m=1}^{k-1}\left\langle
\mathbf{m}^{\left(  k\right)  },\mathbf{\phi}^{\left(  m\right)
}\right\rangle ^{2}\right] \\
\mathbb{E}\left\Vert \mathbf{g}^{\left(  k\right)  }\mathbf{m}^{\left(
k\right)  }\right\Vert ^{2}  & =\mathbb{E}\left[  \left\Vert \mathbf{m}%
^{\left(  k\right)  }\right\Vert ^{2}-\sum\nolimits_{m=1}^{k-1}\left\langle
\mathbf{m}^{\left(  k\right)  },\mathbf{\phi}^{\left(  m\right)
}\right\rangle ^{2}\right]  .
\end{align*}
By Proposition \ref{Prop_inner_products}, the rhs converges, as $N\rightarrow
\infty,$ to $q-\sum_{m=1}^{k-1}\gamma_{m}^{2},$ which converges to $0,$ as
$k\rightarrow\infty,$ by (\ref{AT}) and Lemma \ref{Le_sequences&AT}.
\end{proof}

\begin{lemma}
\label{Le_means_h}For any function $f:\mathbb{R\rightarrow R}$ which satisfies
$\left\vert f\left(  x\right)  \right\vert \leq C\left(  1+\left\vert
x\right\vert \right)  $ for some $C,$ and with $\left\Vert f\right\Vert
_{\mathrm{lip}}<\infty,$ and any $k\geq2,$ one has%
\[
\lim_{N\rightarrow\infty}\frac{1}{N}\sum_{i=1}^{N}f\left(  h_{i}^{\left(
k\right)  }\right)  =Ef\left(  h+\beta\sqrt{q}Z\right)
\]
in $L_{1}.$
\end{lemma}

\begin{proof}
For $k=2,$ this is immediate from the definition of $\mathbf{h}^{\left(
2\right)  }$ and Lemma \ref{Le_chi} a). So, we assume $k\geq3.$ Conditionally
on $\mathcal{G}_{k-2},$ $\mathbf{g}^{\left(  k-1\right)  }\mathbf{m}^{\left(
k-1\right)  }$ is Gaussian with the covariances given in (\ref{gkmk_cov}). For
abbreviation, write%
\[
Y_{i}^{\left(  t\right)  }:=h+\beta\sum_{s=1}^{t}\gamma_{s}\zeta_{i}^{\left(
s\right)  }%
\]
As $\left\Vert \mathbf{m}^{\left(  k\right)  }\right\Vert ^{2}$ and
$\left\langle \mathbf{m}^{\left(  k\right)  },\mathbf{\phi}^{\left(  m\right)
}\right\rangle ^{2}$ are bounded (by $1$), it follows from Lemma \ref{Le_chi}
a) that%
\[
\frac{1}{N}\sum_{i=1}^{N}\left[  f\left(  h_{i}^{\left(  k\right)  }\right)
-\mathbb{E}_{k-2}f\left(  h_{i}^{\left(  k\right)  }\right)  \right]
\rightarrow0
\]
in $L_{1},$ as $N\rightarrow\infty,$ and using Proposition
\ref{Prop_inner_products}, one gets%
\[
\frac{1}{N}\sum_{i=1}^{N}\left[  \mathbb{E}_{k-2}f\left(  h_{i}^{\left(
k\right)  }\right)  -Ef\left(  Y_{i}^{\left(  k-2\right)  }+\sqrt
{q-\sum\nolimits_{s=1}^{k-2}\gamma_{s}^{2}}Z\right)  \right]  \rightarrow0
\]
Next, in the same way, one obtains%
\begin{multline*}
\frac{1}{N}\sum_{i=1}^{N}\Big [Ef\left(  Y_{i}^{\left(  k-2\right)  }%
+\sqrt{q-\sum\nolimits_{s=1}^{k-2}\gamma_{s}^{2}}Z\right) \\
-Ef\left(  Y_{k-3}+\gamma_{k-2}Z_{k-2}+\sqrt{q-\sum\nolimits_{j=1}^{k-2}%
\gamma_{j}^{2}}Z\right)  \Big ]\rightarrow0.
\end{multline*}
Going on in the same way, and observing that%
\[
\sum_{s=1}^{k-2}\gamma_{s}Z_{s}+\sqrt{q-\sum\nolimits_{j=1}^{k-2}\gamma
_{j}^{2}}Z
\]
is identical in law as $\sqrt{q}Z,$ one gets%
\[
\lim_{N\rightarrow\infty}\mathbb{E}\left\vert \frac{1}{N}\sum\nolimits_{i=1}%
^{N}f\left(  h_{i}^{\left(  k\right)  }\right)  -Ef\left(  h+\beta\sqrt
{q}Z\right)  \right\vert =0.
\]

\end{proof}

\begin{lemma}
\label{Le_zeta}

\begin{enumerate}
\item[a)] For any $k$%
\[
\sup_{N}\mathbb{E}\left\Vert \mathbf{\zeta}^{\left(  k\right)  }\right\Vert
^{4}<\infty.
\]

\item[b)] For any $k$ and $\varepsilon>0$%
\[
\lim_{N\rightarrow\infty}\mathbb{P}\left(  \left\Vert \mathbf{\zeta}^{\left(
k\right)  }\right\Vert ^{2}\geq1+\varepsilon\right)  =0.
\]

\item[c)] For $s\neq k,$ and $\varepsilon>0$%
\[
\lim_{N\rightarrow\infty}\mathbb{P}\left(  \left\vert \left\langle
\mathbf{\zeta}^{\left(  k\right)  },\mathbf{\zeta}^{\left(  s\right)
}\right\rangle \right\vert \geq\varepsilon\right)  =0
\]

\end{enumerate}
\end{lemma}

\begin{proof}
By the conditional covariances of $\mathbf{\zeta}^{\left(  k\right)  }$ given
in Proposition \ref{Prop_main_gk} c), $\mathbf{\zeta}^{\left(  k\right)  }$
has, conditionally on $\mathcal{G}_{k-1},$ the covariance structure of
$\mathbf{X}$ in Lemma \ref{Le_chi}. a) and b) of the present lemma are then
immediate from b) of Lemma \ref{Le_chi}.

For c), we assume $s<k.$ Then $\left\langle \mathbf{\zeta}^{\left(  k\right)
},\mathbf{\zeta}^{\left(  s\right)  }\right\rangle $ is Gaussian, conditioned
on $\mathcal{F}_{k-1},$ with conditional variance%
\[
\mathbb{E}_{k-1}\left\langle \mathbf{\zeta}^{\left(  k\right)  }%
,\mathbf{\zeta}^{\left(  s\right)  }\right\rangle ^{2}=s_{N}^{2}\left(
k,s\right)  :=\frac{1}{N}\left[  \left\Vert \zeta^{\left(  s\right)
}\right\Vert ^{2}+\left\langle \mathbf{\zeta}^{\left(  s\right)
},\mathbf{\phi}^{\left(  k\right)  }\right\rangle ^{2}-\sum_{u=1}%
^{k-1}\left\langle \mathbf{\zeta}^{\left(  s\right)  },\mathbf{\phi}^{\left(
u\right)  }\right\rangle ^{2}\right]  ,
\]
and therefore%
\[
\mathbb{P}\left(  \left\vert \left\langle \mathbf{\zeta}^{\left(  k\right)
},\mathbf{\zeta}^{\left(  s\right)  }\right\rangle \right\vert \geq
\varepsilon\right)  =2\mathbb{E}\left(  1-\Phi\left(  \frac{\varepsilon}%
{s_{N}\left(  k,s\right)  }\right)  \right)  ,
\]
$\Phi$ being the distribution function of the standard Gaussian distribution.
Estimating $s_{N}^{2}\leq2\left\Vert \zeta^{\left(  s\right)  }\right\Vert
^{2}/N,$ and using again Lemma \ref{Le_chi} b) proves that the rhs goes to $0$
for $N\rightarrow\infty.$
\end{proof}

\begin{lemma}
\label{Le_gamma_approx}For any $n\geq2$%
\[
\lim_{N\rightarrow\infty}\left\langle \mathbf{\zeta}^{\left(  n-1\right)
},\mathbf{m}^{\left(  n\right)  }\right\rangle =\beta\left(  1-q\right)
\sqrt{q-\sum\nolimits_{j=1}^{m-2}\gamma_{j}^{2}},
\]
and for $1\leq m\leq n-2$%
\[
\lim_{N\rightarrow\infty}\left\langle \mathbf{\zeta}^{\left(  m\right)
},\mathbf{m}^{\left(  n\right)  }\right\rangle =\beta\gamma_{m}\left(
1-q\right)
\]
in $L_{2}\left(  \mathbb{P}\right)  .$
\end{lemma}

\begin{proof}
This is very similar to the proof of Lemma \ref{Le_means_h} and we will be
brief. The case $n=2$ is straightforward, and we assume $n\geq3$%
\[
\mathbf{m}^{\left(  n\right)  }=\operatorname*{Th}\left(  \sum\nolimits_{j=1}%
^{n-2}\gamma_{j}\mathbf{\zeta}^{\left(  j\right)  }+\mathbf{g}^{\left(
n-1\right)  }\mathbf{m}^{\left(  n-1\right)  }\right)  .
\]
Using Lemma \ref{Le_basic1}, we have%
\[
\mathbf{g}^{\left(  n-1\right)  }\mathbf{m}^{\left(  n-1\right)  }=\left\Vert
\mathbf{m}^{\left(  n-1\right)  }-\sum\nolimits_{j=1}^{n-2}\left\langle
\mathbf{m}^{\left(  n-1\right)  },\mathbf{\phi}^{\left(  j\right)
}\right\rangle \mathbf{\phi}^{\left(  j\right)  }\right\Vert \mathbf{\zeta
}^{\left(  n-1\right)  }%
\]
and by Proposition \ref{Prop_inner_products}, we can replace (in the
$N\rightarrow\infty$ limit) the above norm by $\sqrt{q-\sum_{j=1}^{n-2}%
\gamma_{j}^{2}}.$ Therefore $\left\langle \mathbf{\zeta}^{\left(  n-1\right)
},\mathbf{m}^{\left(  n\right)  }\right\rangle $ behaves in the $N\rightarrow
\infty$ limit similarly to%
\[
\left\langle \mathbf{\zeta}^{\left(  n-1\right)  },\operatorname*{Th}\left(
\sum\nolimits_{j=1}^{n-1}\gamma_{j}\mathbf{\zeta}^{\left(  j\right)  }%
+\sqrt{q-\sum\nolimits_{j=1}^{n-2}\gamma_{j}^{2}}\mathbf{\zeta}^{\left(
n-1\right)  }\right)  \right\rangle .
\]
Arguing in the same way is in the proof of Lemma \ref{Le_means_h}, one sees
that this converges to%
\begin{align*}
& EZ\operatorname*{Th}\left(  \sum\nolimits_{j=1}^{n-1}\gamma_{j}Z_{j}%
+\sqrt{q-\sum\nolimits_{j=1}^{n-2}\gamma_{j}^{2}}Z\right) \\
& =\beta\sqrt{q-\sum\nolimits_{j=1}^{n-2}\gamma_{j}^{2}}\left(
1-E\operatorname*{Th}\nolimits^{2}\left(  \sum\nolimits_{j=1}^{n-1}\gamma
_{j}Z_{j}+\sqrt{q-\sum\nolimits_{j=1}^{n-2}\gamma_{j}^{2}}Z\right)  \right) \\
& =\beta\left(  1-q\right)  \sqrt{q-\sum\nolimits_{j=1}^{n-2}\gamma_{j}^{2}},
\end{align*}
the first equality by Gaussian partial integration. The case where $m\leq n-2
$ is going by the same argument, but where we get from partial integration
$\gamma_{m}$ instead of $\sqrt{q-\sum\nolimits_{j=1}^{n-2}\gamma_{j}^{2}}.$
\end{proof}

\section{Comments}

There are a number of issues and open problems we shortly want to comment on.

\noindent\textbf{On the first moment evaluation: }The key idea proposed here
is to derive the free energy by a conditionally annealed argument, where the
$\sigma$-field for the conditioning is chosen such that the solutions of the
TAP equations are measurable. This can reasonably only be done by an
approximating sequence $\mathbf{m}^{\left(  k\right)  }$ for the TAP
equations, where for fixed $k$ one lets first $N\rightarrow\infty,$ and
afterwards $k\rightarrow\infty.$ For finite $N,$ the TAP equations are not
exactly valid, and we wouldn't know how to characterize $\left\langle
\sigma_{i}\right\rangle $ for finite $N$ without knowing the Gibbs measure
already precisely. Therefore, it would be natural just to condition with
respect to $\sigma\left(  \mathbf{m}^{\left(  k\right)  }\right)  ,$ and try
to prove the corresponding versions of Proposition \ref{Prop_1st_moment} and
\ref{Prop_2nd_moment}. We however didn't see how to do this, and therefore, we
took the $\sigma$-fields, generated by $\mathbf{\zeta}^{\left(  s\right)
},\ s\leq k,$ with respect to which $\mathbf{m}^{\left(  k\right)  }$ is
measurable. This choice may well be \textquotedblleft too
large\textquotedblright, in particular as the $\mathbf{\zeta}^{\left(
s\right)  }$ depend on the starting version of $\mathbf{m}^{\left(  1\right)
}$ which we took just as $\sqrt{q}.$ On the other hand, taking $\sigma$-fields
which are larger than necessary should not do any harm for proving Proposition
\ref{Prop_1st_moment}, except that the computations may become unnecessarily
complicated. Anyway, assuming that the replica symmetric solution is valid in
the full AT-region, it looks to me that (\ref{first_moment}) should be correct
in the full AT-region. This belief is based on the hope that the Morita type
argument could give the evaluation in the full high-temperature region. This
hope is also substantiated by the recent work by Jian Ding and Nike Sun
\cite{DingSun} who, for the Ising perceptron, obtained a one-sided (and partly
computer assisted) result in the full replica symmetric region, based on a
method which is related to ours.

Even if our conjecture is correct, there remains the issue how to prove it,
and in particular, whether our choice of the $\sigma$-fields is the best one.
As remarked before, there is nothing lost till (\ref{CurieWeiss}): The region
for $\left(  \beta,h\right)  $ where (\ref{CurieWeiss}) is correct is exactly
the region where (\ref{first_moment}) is correct. (\ref{CurieWeiss}) is a
standard large deviation problem with a Hamiltonian which is of ordinary
mean-field type. In principle, it is not difficult to write down a variational
formula for%
\[
\lim_{N\rightarrow\infty}\frac{1}{N}\log\sum_{\sigma}p^{\left(  k\right)
}\left(  \mathbf{\sigma}\right)  \exp\left[  N\beta F_{N,k}^{\prime
\prime\prime}\left(  \mathbf{\sigma}\right)  \right]
\]
or its $\mathbb{E}$-expectation, and then try to evaluate the $k\rightarrow
\infty$ limit. I have not been able to do that in the full high temperature
region, but, it doesn't appear being impossible. It would be interesting to
clarify this point. It is possible that the above limit is $0$ even beyond the
AT-line, but of course the AT-condition was used to prove that
(\ref{CurieWeiss}) is equivalent to (\ref{first_moment}).

\noindent\textbf{The second moment: }Regardless what the outcome for the first
moment is, I wouldn't expect that the plain vanilla second moment estimate
used here would work in the full high temperature regime. This disbelief is
based on a simple computation for the following toy model: Take $\sigma_{i}$
i.i.d. $\left\{  -1,1\right\}  $-valued with mean $m\in\left(  -1,1\right)
,\ m\neq0,$ and consider the spin glass model with partition function%
\[
Z_{N,\beta,m}:=\sum_{\mathbf{\sigma}}p\left(  \mathbf{\sigma}\right)
\exp\left[  \frac{\beta}{\sqrt{2}}\sum_{i,j}g_{ij}\hat{\sigma}_{i}\hat{\sigma
}_{j}-\frac{\beta^{2}N}{4}\left(  \frac{1}{N}\sum\nolimits_{i}\hat{\sigma}%
_{i}^{2}\right)  ^{2}\right]  ,
\]
where $\hat{\sigma}_{i}:=\sigma_{i}-m.$ If $m=0,$ then the second summand in
the exponent is $\beta^{2}N/4,$ and we have the standard SK-model at $h=0.$ Of
course, $\mathbb{E}Z_{N,\beta,h}=0,$ and the question is whether%
\begin{equation}
\lim_{N\rightarrow\infty}\frac{1}{N}\log Z_{N,\beta,m}=0.\label{toy_model}%
\end{equation}
This is certainly correct for small enough $\beta,$ as can for small $\beta,$
easily be proved by a second moment computation. Indeed,%
\[
\mathbb{E}Z_{N,\beta,m}^{2}=\sum_{\eta}\bar{p}\left(  \eta\right)  \exp\left[
\frac{N\beta^{2}}{2}\left(  N^{-1}\sum\nolimits_{i}\eta_{i}\right)
^{2}\right]  ,
\]
where the $\eta_{i}$ under $\bar{p}$ are i.i.d., and have the distribution of
$\left(  \sigma_{i}-m\right)  \left(  \sigma_{i}^{\prime}-m\right)  $ where
$\sigma_{i},\sigma_{i}^{\prime}$ are independent with distribution $p.$
Therefore%
\[
\lim_{N\rightarrow\infty}\frac{1}{N}\log\mathbb{E}Z_{N,\beta,m}^{2}=\sup
_{x}\left(  \beta^{2}x^{2}/2-J\left(  x\right)  \right)  ,
\]
$J$ being the standard rate function for $\bar{p}.$ It is easily checked that
the right hand side is $0$ for small $\beta,$ and as the second derivative of
$J$ at $0$ is $\left(  1-m^{2}\right)  ^{-2}/2,$ one would expect that this is
true as long as $\beta^{2}\left(  1-m^{2}\right)  ^{2}\leq1.$ That looks to be
the right de Almeida-Thouless condition. However, one easily checks that
$\sup_{x}\left(  \beta^{2}x^{2}/2-J\left(  x\right)  \right)  >0$ for $\beta$
sufficiently close but smaller than $\left(  1-m^{2}\right)  ^{-1},$ for any
choice of $m\neq0,$ a fact which is due to the non-vanishing third derivative
of $J$ at $0.$ Therefore, (\ref{toy_model}) cannot be proved with a simple
second moment computation up to the \textquotedblleft
natural\textquotedblright\ AT-condition. Actually, I don't know if
(\ref{toy_model}) is correct under $\beta^{2}\left(  1-m^{2}\right)  ^{2}%
\leq1.$ (If not already known, it could be a level-$2$-problem in Talagrand's
difficulty scale).

The computation in this toy case suggests that a simple second moment
estimate, in our asymmetric situation when $h\neq0,$ is not sufficient to
cover the full high temperature regime. 

\noindent\textbf{Gibbs distributions: }It is suggestive to conjecture that the
Gibbs distribution (\ref{Gibbs}), in high temperature, is somehow close to the
conditional annealed measure, i.e. the measure on $\Sigma_{N}$ defined by%
\[
\frac{\mathbb{E}_{k}\exp\left[  H_{\beta,h}\left(  \mathbf{\sigma}\right)
\right]  }{\sum_{\mathbf{\sigma}}\mathbb{E}_{k}\exp\left[  H_{\beta,h}\left(
\mathbf{\sigma}\right)  \right]  }%
\]
which, according to the analysis given in this paper, is a kind of complicated
random Curie-Weiss type model, with the centering of the $\mathbf{\sigma}$
given at the solution of the TAP-equation. If correct, this would suggest that
the finite $N$ high temperature Gibbs distributed can be approximated by
random Curie-Weiss models, with however infinitely (if $k\rightarrow\infty$)
many random quadratic interaction terms.

\noindent\textbf{Low temperature: }A main problem is to extend the method to
low temperature. There are many results in the physics literature about the
validity of the TAP equations in low temperature, see \cite{MPV}%
,\ \cite{Plefka}, but the plain iteration method in \cite{BoTAP} is certainly
not able to catch such solutions. However, it has recently been shown by Marc
M\'{e}zard that a similar iterative scheme for Hopfield model converges in the
retrieval phase of the model (see \cite{MezardTAP}). The approximate validity
of the TAP equations in generic $p$-spin models has recently been shown in
\cite{AufJag}. See also the results of \cite{BelKi} and \cite{ChenPanch} on
the TAP variational problem, and \cite{Subag} on the $p$-spin spherical model
where the TAP equations in the full temperature regime are discussed. These
results (except \cite{BelKi}) depend on already having a rather detailed
picture of the Gibbs distribution, whereas the attempt here is to present a
new viewpoint.


\begin{thebibliography}{99}                                                                                               %
\bibitem {AufJag}Auffinger, A., and Jagannath, A.:
\textit{Thouless-Anderson-Palmer equations for conditional Gibbs measures in
the generic p-spin glass model. }arXiv:1612.06359, to appear in Ann. Prob.

\bibitem {BelKi}Belius, D., and Kistler, N.: \textit{The TAP-Plefka
variational principle for the spherical SK model. }arXiv:1802.05782

\bibitem {BoTAP}Bolthausen, E.: \textit{An iterative construction of solutions
of the TAP equations for the Sherrington-Kirkpatrick model. }Comm. Math.
Phys., \textbf{325}, 333--366 (2014).

\bibitem {ChenPanch}Chen, W.-K., and Panchenko, D.: \textit{On the TAP free
energy in the mixed p-spin models. }Comm. Math. Phys. \textbf{362}, 219--252 (2018).

\bibitem {DingSun}Ding, J., and Sun, N.: \textit{Capacity lower bound for the
Ising perceptron. }to appear in arXiv.

\bibitem {GuTo}Guerra, F. and Toninelli, F. L.: \textit{The thermodynamic
limit in mean field spin glass models. }Comm. Math. Phys. \textbf{230}, 71-79 (2002).

\bibitem {MPV}M\'{e}zard, M., Parisi, G., and Virasoro, M.A.: \textit{Spin
glass theory and beyond.} World Scientific LN in Physics, Vol 9. World
Scientific 1987.

\bibitem {MezardTAP}M\'{e}zard, M.: \textit{Mean-field message-passing
equations in the Hopfield model and its generalizations. }Phys. Rev. E
\textbf{95, }22117-22132 (2017).

\bibitem {Morita}Morita,T.: \textit{Statistical mechanics of quenched solid
solutions with application to magnetically dilute alloys.} J. Math. Phys.
\textbf{5}, 1401-1405 (1966).

\bibitem {PanchBook}Panchenko, D.: \textit{The Sherrington-Kirkpatrick model.
}Springer, New York, 2013.

\bibitem {Plefka}Plefka, T.: \textit{Convergence condition of the TAP equation
for the infinite-ranged Ising spin glass model. }J. Phys. A: Math. Gen.
\textbf{15}, 1971--1978 (1982).

\bibitem {SK}Sherrington, D., and Kirkpatrick, S.: \textit{Solvable model of a
spin-glass.} Phys. Rev. Lett. \textbf{35}, 1792--1795 (1975).

\bibitem {Subag}Subag, E.: \textit{Free energy landscapes in spherical spin
glasses. }arXiv:1804.10576

\bibitem {TalaBuch}Talagrand, M.: \textit{Mean field models for spin glasses.
Volume I\&II. }Springer, Berlin, 2011.

\bibitem {TAP}Thouless, D.J., Anderson, P.W., and Palmer, R.G.:
\textit{Solution of }\textquotedblleft\textit{solvable model in spin
glasses}\textquotedblright. Philosophical Magazin \textbf{35, }593-601 (1977).
\end{thebibliography}
\end{document}